\documentclass{amsart}
 
\usepackage{amssymb}
\usepackage{enumitem}
\usepackage{xparse}
\usepackage[unicode,psdextra]{hyperref}

\usepackage[steps,norsfs]{rec-thy} 

\ifpdftex
    \usepackage{mathrsfs}
\else 
    \usepackage{unicode-math}
    \usepackage{fontspec}
\fi

\usepackage{hyperxmp}  
\hypersetup{pdfauthor={Peter M. Gerdes},
            pdfsubject={Arithmetic Degrees},
            pdfkeywords={Computability Theory, Arithmetic reducibility, Implicit definability, Minimal degree},
            pdftitle={A \textPi⁰₂ Singleton of Minimal Arithmetic Degree},
            pdflang={en},
            unicode=true,
            pdfdisplaydoctitle=true,
            pdfencoding=unicode,
            psdextra=true,
            pdfcontactemail={gerdes@invariant.org},
            pdfcontacturl={http://invariant.org},
            pdfpublication={Journal of Symbolic Logic},
            pdfcontactaddress={Department of Mathematics,
                                \xmpquote{Indiana University\xmpcomma Bloomington}},
            pdfcontactcity={Bloomington},
            pdfcontactcountry={United States of America},
            pdfeissn={19435886}, 
            pdfissn={00224812}, 
            pdfpubstatus={AM}, 
}

\theoremstyle{plain}
\newtheorem{theorem}{Theorem}[section]
\newtheorem{proposition}[theorem]{Proposition}
\newtheorem{lemma}[theorem]{Lemma}
\newtheorem{corollary}[theorem]{Corollary}
\theoremstyle{definition}
\newtheorem{definition}[theorem]{Definition}

\newtheorem{condition}{Condition}
\theoremstyle{remark}

\newtheorem{question}[theorem]{Question}

\theoremstyle{plain}
\newtheorem*{minsingleton}{Corollary \ref{cor:minimal-singleton}}
\newtheorem*{thm:arith-min-w-tree}{Theorem \ref{thm:arith-min-w-tree}}
\newtheorem*{prop:tree-build}{Proposition \ref{prop:tree-build}}
\newtheorem*{prop:no-do-one-jump}{Proposition \ref{prop:no-do-one-jump}}

\let\pTree=\PriorityTree

\NewDocumentCommand{\axset}{D<>{}O{}D<>{#1}d()}{\mathcal{A}^{#3}_{#2}\IfValueTF{#4}{\left(#4\right)}{}}

\let\hat=\widehat

\newcommand*{\copyord}[1]{{#1}^{\Diamond}}
\newcommand*{\copylen}[1]{l^{\Diamond}(#1)}
\newcommand*{\Msucc}{\mathfrak{O}}
 \NewDocumentCommand{\hT}{t-od()}{\hat{T}\IfBooleanTF{#1}{^{-}}{}\IfValueTF{#2}{_{#2}}{}\IfValueTF{#3}{\!\left(#3\right)}{}}
 \newcommand*{\Mout}[1]{\delta^{#1}}
\newcommand*{\Mexts}[1]{\Theta^{#1}}
\newcommand*{\Mext}[2][]{\theta^{#2}_{#1}}
\newcommand*{\subtree}[2]{#1^{\left\langle #2 \right\rangle}}
 \newcommand*{\frciseg}[3][]{\psi_{#2}^{\forces[#1]}\left(#3\right)}

\title{A \texorpdfstring{\( \pizn{2} \)}{\textPi⁰₂} Singleton of Minimal Arithmetic Degree}
\keywords{computability theory,arithmetic degrees,arithmetic reducibility,minimal degree,singleton,REA,arithmetic singleton,implicit definability,Turing reducibility}
\subjclass[2020]{Primary 03D28, 03D30; Secondary (03D60)}
\author{Peter M. Gerdes}
\address{Indiana University, Bloomington\\
Department of Mathematics\\
Bloomington, IN}
\email{gerdes@invariant.org}
\urladdr{http://invariant.org}

\begin{document}

\begin{abstract}
In the study of the arithmetic degrees the \( \REA[\omega] \) sets play a role analogous to the role the r.e. degrees play in the study of the Turing degrees.  However, much less is known about the arithmetic degrees and the role of the  \( \REA[\omega] \) sets in that structure than about the Turing degrees.  Indeed, even basic questions such as the existence of a \( \REA[\omega] \) set of minimal arithmetic degree are open.  This paper makes progress on this question by demonstrating that some promising approaches inspired by the analogy with the r.e. sets fail to show that no \( \REA[\omega] \) set is arithmetically minimal.  Finally, it constructs a \( \pizn{2} \) singleton of minimal arithmetic degree.  Not only is this a result of considerable interest in it's own right, constructions of \( \pizn{2} \) singletons often pave the way for constructions of  \( \REA[\omega] \) sets with similar properties.  Along the way, a number of interesting results relating arithmetic reducibility and rates of growth are established.
\end{abstract}

\maketitle

\section{Introduction}

In the study of the arithmetic degrees (the degree structure induced by relative arithmetic definability, \( \Aleq \)) the \( \REA[\omega] \) sets play a role analogous to the role the r.e. degrees play in the study of the Turing degrees.  This analogy holds both as a matter of structure (e.g. the arithmetic jump is an \( \REA[\omega] \) operation and the jump is an r.e. operation) and as a way to approach constructions.  For instance, just as the r.e. sets allow us to characterize the range of the Turing jump on \( \Tdegrees(\zeroj) \) (the Turing degrees less than or equal to \( \zeroj \)) via the Shoenfield jump inversion \cite{Shoenfield1959On} the \REA[\omega] sets allow us to similarly identify the range of the arithmetic jump on \( \Adegrees(\Azeroj) \) with the degrees of the sets \( \REA[\omega] \) in \( \Azeroj \) \cite{Simpson1985Arithmetic}.

However, while the Turing degrees generally and the degrees of the r.e. sets specifically have been extensively studied much less is known about the arithmetic degrees and even less about the role of the \( \REA[\omega] \) sets in that structure.  Even seemingly basic questions remain open.  For instance, whether or not there are any \( \REA[\omega] \) sets of minimal arithmetic degree remains an open question.  The analogy between the r.e. sets and the Turing degrees and the \( \REA[\omega]  \) sets suggests that \( \REA[\omega] \) sets of minimal arithmetic degree shouldn't exist.  However, it is already known that the analogy is imperfect as there is a minimal pair (in the arithmetic degrees) of \( \REA[\omega] \) sets which join to \( \Azeroj \) \cite{Simpson1985Arithmetic} in contrast to the non-diamond theorem in the r.e. degrees \cite{Lachlan1966Lower}.  

While we don't settle the existence of an \( \REA[\omega] \) set of minimal arithmetic degree in this paper we make what we believe is an important step in that direction by proving the following result (and thereby presenting an alternate solution to question 62 in \cite{Friedman1975One} by providing a \( \pizn{2} \) singleton not arithmetically equivalent to any \( \zeron{\alpha} \)). 

\begin{minsingleton}
    There is a \( \pizn{2} \) singleton of minimal arithmetic degree. 
\end{minsingleton}

While the degrees of \( \REA[\omega] \) sets are properly contained in the degrees of singletons, results about \( \pizn{2} \) singletons have often paved the way for corresponding results about \( \REA[\omega] \) sets, e.g., Harrington's construction of an arithmetically low \( \pizn{2} \) singleton \cite{Harrington1976MclaughlinS} or arithmetically incomparable \( \pizn{2}\) singletons \cite{Harrington1976Arithmetically} both\footnote{These are both unpublished notes that sketch the approach.  See \cite{Simpson2016Implicit} or (draft work) \cite{Gerdes2010HarringtonS} for more rigorous write ups of some of the results.} foreshadowed related constructions of \( \REA[\omega] \) sets with those properties \cite{Simpson1985Arithmetic}.  As the approach taken in this paper draws heavily on the ideas in \cite{Harrington1976MclaughlinS} we hypothesize, contra the analogy with r.e. sets, that there is an arithmetically minimal \( \REA[\omega] \) set and hope the construction here points to a way to build such a set.

\section{Background}

There is a fair amount of notation required for the results in this paper, however, almost all of it is standard.  Readers familiar with standard notation may wish to skip most of the subsections below and return to them only as needed for reference.  However, Section \ref{ssec:rates-of-growth} is worth looking at for all readers as it contains some slightly less common definitions and results.

\subsection{Computations, Strings and Degrees}

We largely adopt standard notation from \cite{Odifreddi1992Classical} which we briefly review.  Set difference is denoted by \( X \setminus Y \), powerset by \( \powset{X} \),   the \( e \)-th set r.e. in \( X \) is  \( \REset(X){e} \) and the \( e \)-th computable functional applied to \( X \) by \( \recfnl{i}{X}{} \).   We denote convergence and divergence by \( \recfnl{i}{X}{y}\conv   \) and \( \recfnl{i}{X}{y}\diverge   \) respectively. Convergence in \( s \)-steps is denoted by \( \recfnl{i}{X}{y}\conv[s] \iff \recfnl[s]{i}{X}{y}\conv    \) and it's negation by \( \recfnl{i}{X}{y}\nconv[s]  \).        

We write elements of \( \bstrs \) and \( \wstrs \) (referred to as strings) like \( \str{x_0, x_1, \ldots, x_{n-1}} \) with \( \estr \) denoting the empty string.   For elements in \( \bstrs, \wstrs, \baire, \cantor \) (identifying sets with their characteristic functions) we denote that \( \sigma \) is (non-strictly) extended by \( \tau \) by \( \sigma \subfun \tau \),  incompatibility by \( \incompat \), compatibility by \( \compat \) and use \( \leftof \) to denote the lexicographic ordering.   For elements of \( \bstrs, \wstrs  \) we denote the length of \( \sigma \) by \( \lh{\sigma} \) and write \( \sigma^{-} \) to indicate the immediate predecessor of \( \sigma \) under \( \subfun \) and write \( \alpha\concat \beta \) to denote \( \alpha \) concatenated with \( \beta \) and write \( \str{i}^n \) to denote \( \str{i} \) concatenated with itself \( n \leq \omega \) times.     We let \( \recfnl{e}{\sigma}{} \) denote the longest string \( \tau, \lh{\tau} \leq \lh{\sigma} \) such that \( \tau(n) = \recfnl[\lh{\sigma}]{e}{\sigma}{n}\conv \) and relativize to define \( \recfnl{e}{\sigma \Tplus X}{} \) in the obvious manner.

We let \( \godelnum{\alpha} \)  denote the canonical bijection of \( \omega^{< \omega} \)  with \( \omega \) where \( \alpha \subfun \beta  \implies \godelnum{\alpha} < \godelnum{\beta} \), \( i < j \implies \godelnum{\alpha\concat[i]} < \godelnum{\alpha\concat[j]} \) and \( \godelnum{\estr} = 0 \).  We regularly gloss over the distinction between strings and their codes as that between sets and their characteristic functions.  We let \( \pair{x}{y} = \frac{1}{2}(x+y)(x+y+1)+y  \) (this is a bijection of \( \omega^2 \) with \( \omega \)) and we let \( \decode{\pair{a}{b}}{0} = a  \) and \( \decode{\pair{a}{b}}{1} = b  \).  We define \( A \Tplus B \),  \( \Tplus_{n \in S} X_n  \), \( \setcol{X}{n} \) and \( \setcol{X}{< n}  \) standardly and extend these operations to strings in the obvious fashion (preserving the identification of \( \cantor \) with \( \powset{\omega} \)).    

A set \( X \) is arithmetic in \( Y \) (written \( X \Aleq Y \)) just if there is a formula in the language of arithmetic  (with a designated set constant)  \( \psi_e \)  such that \( Y \models \psi_e(z) \iff z \in X \) (see \cite{Odifreddi1992Classical} for details).  In this case we write \( \psi_e(Y) = X \).  Recall that an equivalent characterization of arithmetic reducibility is given by \( X \Aleq Y \iff \exists(n)\left(X \Tleq \jumpn{Y}{n} \right) \).   We denote the arithmetic degree of \( \eset \) by \( \Azero \) and, as the arithmetic jump of \( X \) is defined to be \( X^{\omega} \),  that of \( \zeron{\omega} \) by \( \Azeroj \).  An arithmetic degree is minimal just if it has exactly one predecessor under \( \Aleq \).

\subsection{Trees and Forcing}

A tree \( T \) is a \( \subfun \) closed set of strings and \( [T] \) is the set of paths through \( T \).  We define \( \pruneTree{T} = \set{\sigma \in T}{\exists(f \in [T])\left(\sigma \subfun f \right)} \) and call a tree pruned if \( T = \pruneTree{T} \). A node \( \sigma \in T \) is terminal if \( \sigma \) has no extensions in \( T \), branching if it has more than one immediate extension, \( \omega \)-branching if it has infinitely many immediate extensions  and the root of \( T \) if it is the unique \( \subfun \)  least branching node in \( T \).         A tree is \( \omega \)-branching if every branching node is \( \omega \)-branching.  A subtree of \( T \) is a tree \( \hat{T} \subset T \).   

We abuse notation and write \( T\restr{l} \) for \( \set{\sigma \in T}{\lh{\sigma} \leq l}  \).  We write \( \sigma \TreeMul T \) for \(  \set{\sigma\concat\tau}{\tau \in T} \), \( \TreeMod{T}{\sigma}  \) for \(  \set{\tau}{\sigma\concat\tau \in T} \) (with \( \TreeMod{\tau}{\sigma} \) the unique element of \(  \TreeMod{\set{\tau}{}}{\sigma} \) )  and \( \subtree{T}{\sigma} = \sigma \TreeMul \TreeMod{T}{\sigma}  \) called the subtree of \( T \) above \( \sigma \).  
We recall that the standard topology on \( \baire \) is induced by basic open sets  \( [\sigma]  = \set{f}{f \supfun \sigma} \) for \( \sigma \in \wstrs \) and likewise for \( \cantor \).  A set is perfect if it contains no isolated points and we call a tree \( T \) perfect if \( [T] \) is perfect.  

When working with unpruned trees in \( \wstrs \) we can't identify trees (\( \subfun \) closed sets) with \( \subfun \) respecting functions on strings as we'll do over\footnote{In \( \bstrs \) we can, computably in \( T \), find the least extension of \( \sigma \in T \) which either has no extensions in \( T \) or two extensions in \( T \), provided we assume such an extension always exists.} \( \bstrs \). We use the term f-tree for a \( \subfun, \leftof \) respecting partial function  \( T\maps{\wstrs}{\wstrs} \) with a \(  \subfun \) closed domain that preserves longest common initial segments, i.e., \( T(\sigma\concat[n])(\lh{T(\sigma)}) \) is strictly monotonic in \( n \) (on it's domain).  A f-tree \( \hat{T} \) is a subtree of an f-tree \( T \) if \( \rng \hat{T} \subset \rng T \).   An f-tree \( T \) is branching if every non-terminal element in \( \dom T \) has multiple immediate successors, (weakly) \( \omega \)-branching if every non-terminal element in \( \dom T \) has infinitely many immediate successors and  \( \omega \)-branching if whenever \( \sigma \in \dom T \) is non-terminal than \( \sigma\concat[n] \in \dom T \) for all \( n \).  Unless otherwise stated, we assume every f-tree is a branching f-tree and generalize the notions of being pruned, \( \pruneTree{T} \),  \( \TreeMod{T}{\tau} \) , \( \sigma \TreeMul T \) and \( \subtree{T}{\tau}  \) to f-trees in the obvious way.

We write \( \sigma \forces \psi \) to denote forcing over \( \bstrs  \) or \( \wstrs \) and \( \sigma \forces_T \psi \) to denote local forcing on the (pruned) tree \( T \) (see \cite{Odifreddi1992Classical} for details).    A set/function is \( \kappa \)-generic iff it forces either \( \psi \) or \( \lnot \psi \) for every \( \sigmazn{\lambda} \) sentence with \( \lambda < 1 + \kappa \) and weakly \( \kappa \)-generic (note the application at limit ordinals) iff it meets  every dense  \( \sigmazn{\lambda} \) set of strings.   Recall that a set of strings \( W \)  is dense if every string is extended by an element of \( W \).

\subsection{\texorpdfstring{\( \pizn{n} \)}{\textPi⁰n} Classes and \texorpdfstring{\REA[\omega]}{\textomega-REA} sets}

A \( \pizn{n} \) set  class is the set of elements in \( \powset{\omega} \) (identified with \( \cantor \))  that satisfy some  \( \pizn{n} \) formula with a free set  variable.  A \( \pizn{n} \) function  class is defined likewise for elements in \( \baire \).  We will use the term \( \pizn{n} \) class without further specification to refer to a \( \pizn{n} \) set class. We note that if \( n > 0 \) then  \( \mathscr{F} \subset \baire \) is a \( \pizn{n} \) function class iff there is a computable relation \( R \) and a quantifier block \( \forall x \ldots Qy  \) containing \( n \) alternations such that \( f \in \mathscr{F} \)  iff  \(   \forall x \ldots Qy  R(f\restr{y}, x, \ldots, y) \) (and likewise for a \( \pizn{n} \) set class).

An immediate consequence of this fact is that \( \pizn{1} \) classes can be identified with the set of paths through a computable tree.  Interestingly, up to degree, \( \pizn{2} \) classes and \( \pizn{1} \) function classes are equivalent in the following sense.

\begin{lemma}\label{lem:function-to-set}
    Every \( \pizn{1} \) function class is homeomorphic with a \( \pizn{2} \) class via a computable (hence degree preserving) homeomorphism and vice versa.   This holds with all possible uniformity.
\end{lemma}
\begin{proof}
For the first claim, it is enough to note that there is a computable homeomorphism of \( \baire \) with  \( \cantor \setminus \set{\sigma\concat\str{1}^\omega}{\sigma \in \bstrs} \) given by setting  \( \Gamma(\estr) = \estr \) and \( \Gamma(\sigma\concat[i]) = \Gamma(\sigma)\concat\str{1}^i\concat[0] \). 

The other direction is slightly more tricky.  Given a \( \pizn{2} \) class \( \mathscr{C} \) such that \( X \in \mathscr{C} \iff \forall(z)\exists(y)R(X\restr{y}, z, y) \) define \( \Gamma(X)(2z) \) to be the least \( y \) such that \( R(X\restr{y}, z, y) \) and \( \Gamma(X)(2z+1) = X(z) \).  We now define a computable tree \( T \subset \wstrs \) with \( \mathscr{F} = [T] \).  Given \( \sigma \in \wstrs \) with \( \lh{\sigma} \equiv 0 \pmod{2} \) let \( \sigma = \tau \Tplus \epsilon\) and place \( \sigma \in T \) iff \( \forall(l < \lh{\tau})R(\epsilon\restr{l}, \tau(l), l) \).  If \( \lh{\sigma} \equiv 1 \pmod{2} \) then place \( \sigma \in T \) iff either \( \sigma\concat[0] \in T \) or \( \sigma\concat[1] \in T \).  Clearly, \( T \) is a computable tree and \( \Gamma \) is a computable continuous bijection of \( \mathscr{C} \) with \( [T] \) with a continuous inverse on \( [T] \).       
\end{proof}

As the name would suggest, a \( \pizn{n} \) singleton is a set/function that's the only element in a \( \pizn{n} \) class.  Recall that every \( \REA[\omega] \) set is a \( \pizn{2} \) singleton but not vice-versa.  For the interested reader unfamiliar with the \REA[\omega] sets we refer them to \cite{Jockusch1984PseudoJump} but as these sets will primarily play a motivating role in this paper it's enough for the reader to understand that they are the result of effectively iterating the operation  \( X \mapsto  X \Tplus \REset(X){i} \) \( \omega \) many times  (so the \( n +1 \)-st component must be uniformly r.e. in the first \( n \) components).     

\subsection{Ordinal Notations}

We will generalize our main theorem past \( \omega \) to arbitrary ordinal notations.  The reader interested in only claims about arithmetic reductions (such as the headline corollary) can assume that notations only range over \( \omega \union \set{\omega}{} \) as the lemmas required for larger ordinal notations have been exiled to Appendix \ref{app:ordnottech}.  However, some notational conventions, inspired by \cite{Sacks1990Higher}, are still necessary in either case.

Kleene's set of ordinal notations is \( \kleeneO\) with ordering \( \Oless \).  The height of \( \kappa \) is \( \Oabs{\kappa} \).  \( \kleeneO- \) is the set of limit notations, \( \kleeneO+ \) the set of successor notations.    For \( \lambda \) a limit notation we denote the \( n \)-th element of the effectively given increasing sequence defining \( \lambda \)  by \( \Ofunc{\lambda}(n) \).   We elide the differences between finite notations and elements of \( \omega \) as well as that between \( \omega \) and some canonical notation for it.

\subsection{Rates of Growth}\label{ssec:rates-of-growth}

We say that \( g \in \baire \) is \( C \)-escaping  if \( g \) isn't dominated by any \( f \in C \),  \( X \)-escaping for \( X \subset \omega \) if it escapes from \( \set{f}{f \Tequiv X} \) and arithmetically escaping if it escapes from the set of arithmetic functions.  Recall that \( f \) majorizes (dominates)  \( g  \) iff  \(f(x) \geq g(x) \) for  all \( x \) (all but finitely many \( x \)). Following \cite{Odifreddi1992Classical} we draw on the fact that a Turing degree is hyperimmune just if it contains a \( \Tzero \)-escaping function and say that an arithmetic degree is arithmetically hyperimmune just if it contains an arithmetically escaping function.  As there is no notion of a set being arithmetically hyperimmune to cause confusion, we call \( X \) arithmetically hyperimmune just if it is of arithmetically hyperimmune degree.  

It's a well-known fact that if \( A \) is an r.e. set then \( A \) is uniformly computable in any \( g \) majorizing \( m_A(x) = \min \set{s}{A\restr{x} = A_s\restr{x}} \).  Thus, the degree of an r.e. set can be characterized in terms of the rate of growth of a function computable in that degree.   However, it's slightly less well-known that this isn't just true of r.e. sets but of all \( \pizn{2} \) singletons. 

\begin{lemma}\label{lem:uniform-modulus}
If \( f \) is a \( \pizn{1} \) function singleton then every \( g \) majorizing \( f \) uniformly computes \( f \).  If \( X \) is a \( \pizn{2} \) singleton then there is some \( f \Tequiv X \) such that every \( g \) majorizing \( f \) uniformly computes \( X \).   Moreover, if \( X \) isn't arithmetic than it is arithmetically hyperimmune.       
\end{lemma}
\begin{proof} 
By Lemma \ref{lem:function-to-set} the second and third claims follow from the first.  Now suppose that \( f \) is the unique path through \( [T] \) and that \( f \) is majorized by some arithmetic function \( g \).  Now let \( \hat{T} \) consist of all \( \sigma \in T \) majorized by \( g \) (on \( \dom \sigma \)).  To compute \( f\restr{l} \) from \( g \), search for some \( \sigma, k \) such that \( \lh{\sigma} = l \) and \( \forall(\tau \in \hat{T})\left(\lh{\tau} = k \implies \tau \supfun \sigma \right)  \).  As \( \hat{T} \) is finitely branching and \( [\hat{T}] = \set{f}{} \)  K\"onig's lemma guarantees that such a \( k \) exists.                
\end{proof}

\section{Fast Growing Functions and Minimality}\label{sec:fast-growth-min-arith}

Before we prove Corollary \ref{cor:minimal-singleton} we first consider a seemingly promising, but ultimately futile, approach to proving that no \REA[\omega] set can be of minimal arithmetic degree.  This failure provides  an interesting result about the arithmetic degrees in its own right and illustrates both some of the similarities and differences between the role of the r.e. sets in the Turing degrees and the \REA[\omega] sets in the arithmetic degrees.  As an added bonus it will preview some issues that will arise later and remind readers of the standard construction of a minimal arithmetic degree \cite{Sacks1971Forcing}.

\subsection{Motivation}

One of the most powerful methods to prove results about r.e. sets is to threaten to code one set into another (e.g., see Sack's proof of the Density theorem \cite{Sacks1964Recursively}).  However, translating this approach to the \REA[\omega] sets under \( \Aleq \) faces two serious barriers.   First, the fact we can only place elements into the \( n + 1 \)-th component of an \REA[\omega] set if they are enumerated in an r.e. fashion from the \( n \)-th component.  This makes it very difficult to threaten to code \( X \) into \( Y \) without following through.  Second, the coding would somehow have to control/react to facts about arbitrarily many jumps of \( Y \).   

One potential way to avoid these difficulties with coding is to ignore the details about what elements enter a set and just focus on rate of growth/domination.  Every non-arithmetic \( \pizn{2} \) singleton (and hence \( \REA[\omega] \) set) computes an arithmetically escaping function and we know that non-domination strength is often a good way to build sets of smaller degree, e.g., Kurtz's proof that every hyperimmune degree computes a weak \( 1 \)-generic \cite{Kurtz1981Randomness}.   

 Also, \cite{Andrews2014Degrees} showed that every \( \zeroj \)-escaping function computes a weak \( 2 \)-generic (and is thus not of minimal Turing degree).  While further non-domination strength won't ensure we can compute a weak \( 3 \)-generic an examination of the proof of this claim \cite{Andrews2014Degrees}  suggests that this is more about the ease of avoiding genericity not necessarily a limitation on the computational power of non-domination strength.  This leaves open the possibility that arithmetically escaping functions compute \( \omega \)-generics under some local forcing\footnote{However, a modification of the same argument will allow one to show that no amount of non-domination is enough to compute a generic with respect to local forcing on some \textit{sufficiently definable} pruned perfect tree \( T_e \).  The construction is as before, except now we compute a pair \( T_e, X_e \) from each path on \( T \).  We proceed much as before but with the addition that we achieve immediate victory if we can force \( X_e \) to leave \( T_e \).  However, this leaves open the possibility we compute some \( \omega \)-generic relative to local forcing on some more complicated tree.} or other forcing notion.  Besides, it simply seems intuitively unlikely that a minimal arithmetic degree, a degree which should have the very least amount of computational power, could include a arithmetically escaping function.

This possibility is rendered more plausible by the fact that the the standard construction of a minimal arithmetic degree \cite{Sacks1971Forcing} naturally produces a set which doesn't compute any arithmetically escaping functions. 

\begin{proposition}\label{prop:std-mina-no-escaping}
    Suppose that for every \( n \) there is an arithmetic tree \( T_n \subset \bstrs \) such that \( X \in [T_n] \) and every path through \( T_n \) is  \( n \)-generic with respect to \( \forces[T_n] \)  then \( X \) is arithmetically hyperimmune-free. 
\end{proposition}
\begin{proof}
Suppose \( f \in \baire \) and \( X \models \psi(x,y) \iff f(x) = y \) for some \( \sigmazn{n}  \) sentence \( \psi \).  
We now construct an arithmetic \( g \) majorizing \( f \) as follows.  

Let \( m = n + 2 \) and consider the sentence that asserts \( \psi(x,y) \) defines a total function 
\[ \varphi \eqdef \forall(x)\exists(y)\left(\psi(x,y)\right) \land \forall(x)\forall(y)\forall(y')\left(\psi(x,y) \land  \psi(x,y') \implies y = y'\right) \]  
By assumption, \( X \models \varphi \) and therefore there is some \( \sigma \subfun X, \sigma \in T_m \) with \( \sigma \forces[T_m] \varphi  \).  Let \( \hat{T} \) be the strings in \( T_m \) compatible with \( \sigma \).  

We compute \( g(x) \) by searching for a finite set of pairs \( y_i, \tau_i \) with  \( \tau_i \forces[T_m] \psi(x,y_i)  \) such that every path through \( [\hat{T}] \) extends some \( \tau_i \).  Since all \( Y \in [\hat{T}] \) are \( m \) locally generic every such \( Y \models \varphi  \).  Thus, for all \( Y \in [\hat{T}] \) there are \( x, y \) such that  \( Y \forces[T_m] \psi(x,y)  \).  Thus, there is an set of pairs \( y_i, \tau_i \) as described and the \( \tau_i \) may be taken to be incompatible.   As \( T_m \) is finitely branching by  K\"onig's lemma, there can only be finitely many elements in \( \hat{T} \) extending no \( \tau_i \) ensuring that we can always find a finite collection of such pairs.  Now let \( g(x) \)  be larger than all the \( y_i \) in our set of pairs.    

If \( Y \in [\hat{T}] \) then \( g \) majorizes the function \( f_Y  \) where \( f_Y(x) = y \iff Y \models \psi(x,y) \) and as \( X \in [\hat{T}] \) it follows that \( g \) majorizes \( f \).  As \( T_m \Tgeq \hat{T} \) is arithmetic it follows that \( g \) is an arithmetic function majorizing \( f \).        
   
\end{proof}

The trees \( T_n \) in the above proposition track the trees used in the construction of a minimal arithmetic degree.  Thus, the usual construction of a minimal arithmetic degree produces an arithmetically hyperimmune-free degree.

\subsection{An Arithmetically Hyperimmune Minimal Degree}

Unfortunately, despite the reasons to conjecture that arithmetically hyperimmune functions couldn't be of minimal arithmetic degree, it turns out not to be the case.  Indeed, it turns out that any amount of non-domination strength is compatible with being of minimal arithmetic degree.  This contrasts with the situation in the Turing degrees where no \( \zeroj \) escaping function can be of minimal Turing degree.

\begin{theorem}\label{thm:arith-min-w-tree}
There is a pruned perfect \( \omega \)-branching  f-tree \( T \Tleq \zeron{\omega} \) such that every \( f \in [T] \)  is of minimal arithmetic degree.  
\end{theorem}

We will break the proof of this theorem up into a sequence of lemmas.  However, before we do that let's first verify that the theorem actually provides the desired (arguably undesired) arithmetically minimal, arithmetically hyperimmune degree.

\begin{corollary}\label{cor:hi-arith-min}
There is a arithmetically minimal degree \( \mathbf{a} \) that is of arithmetically hyperimmune degree.  Indeed, for any countable \( C \subset \baire \) there is a minimal arithmetic degree \( \mathbf{a} \) containing a \( C \) escaping member.  
\end{corollary}
\begin{proof}
The first claim follows from the second by taking \( C \) to be the collection of arithmetic \( f \in \baire \).  For the second claim, fix a function \( h \)  dominating every element of \( C \) and then build a path \( f \) through the tree in Theorem \ref{thm:arith-min-w-tree} by always extending \( \sigma \) to \( \sigma\concat[m] \) such that \( T(\sigma\concat[m])\conv \) and escapes \( h \) one more time.            
\end{proof}

Recall from the construction of a minimal Turing degree that 

\begin{definition}\label{def:e-splitting}
The strings \( \tau_0, \tau_1 \) \( e \)-split if \( \recfnl{e}{\tau_0}{} \incompat  \recfnl{e}{\tau_1}{} \)   
\end{definition}

and that in that construction we built a sequence of computable trees \( T_e \subset \bstrs \) with \( T_{e+1} \) a subtree of \( T_e \)   such that one of the following obtains (here we identify pruned, perfect binary trees and the corresponding f-trees).

\begin{enumerate}
    \item\label{enum:min-deg-partial} (Partiality) \( \forall(f \in [T_e])\left(\recfnl{e}{f}{}\diverge\right) \)
    \item\label{enum:min-deg-nonsplitting} (Non \( e \)-splitting) for all \( \tau, \tau' \in T_e \), \( \tau, \tau' \) don't \( e \)-split  
    \item\label{enum:min-deg-e-splitting} (\( e \)-splitting) For all \( \sigma \in T_e \), \( T_e(\sigma\concat[0]) \) and  \( T_e(\sigma\concat[1]) \)  \( e \)-split. 
\end{enumerate}

 We then build \( f \in \Isect_{e \in \omega} [T_e] \) ensuring that either \( \recfnl{e}{f}{} \) is partial (\ref{enum:min-deg-partial}), computable (\ref{enum:min-deg-nonsplitting}) or computes \( f \) (\ref{enum:min-deg-e-splitting}).  We adopt the same general approach, but, to handle arithmetic reductions rather than Turing reductions we'll replace the notion of \( e \)-splitting with an analogue based on local forcing (as in the construction of a minimal arithmetic degree from \cite{Sacks1971Forcing}).  We'll then adjust this construction to allow us to build \( \omega \)-branching trees.   First, however, we introduce notation to represent an analogue of partial application of a functional for forcing.

 \begin{definition}\label{def:forcing-iniseg}
 Given a notion of forcing \( \forces \) a condition \( \sigma \) and a sentence \( \psi \) with a single free (number) variable let \( \frciseg{}{\sigma}  \) denote the longest string \( \tau \in \bstrs \) such that \( n \in \dom \tau \) implies \( \sigma \forces \psi(n) \land \tau(n) = 1 \) or  \( \sigma \forces \lnot\psi(n) \land \tau(n) = 0 \).  We extend this in the obvious way to infinite paths.
 \end{definition}

 In other words, \( \psi^{\forces}(\sigma) \) represents the initial segment of  \( \psi(A)  \) whose values have been determined for \( A \supfun \sigma \)  (assuming \( A \) is sufficiently generic).  

 \begin{definition}\label{def:e-forced-splitting}
 If \( \psi_e \) is a \( \sigmazn{n} \) or \( \pizn{n} \) formula with a single free variable then
    \begin{enumerate}
    \item A pair of strings \( \tau, \tau' \) \( e \)-fsplits on \( T \)  just if \( \frciseg[T]{e}{\tau} \incompat \frciseg[T]{e}{\tau} \).
    \item A pruned f-tree \( T \)  is totally non-\( e \)-fsplitting if there are no \( \tau, \tau' \) in the image of \( T \) that \( e \)-fsplit.  
    \item A pruned f-tree is totally \( e \)-fsplitting if whenever \(\sigma \in \wstrs,  n \neq n' \) then \( T(\sigma\concat[n])  \),   \( T(\sigma\concat[n'])  \) \( e \)-fsplits whenever both are defined.  
    \item A pruned f-tree \( T \)  is \( e \)-deciding if it is either totally non-\( e \)-fsplitting or totally \( e \)-fsplitting and every path through \( T \) is \( n \) generic with respect to \( \forces_T \).  
    \end{enumerate} 
 \end{definition} 

 \begin{lemma}\label{lem:e-deciding-nomin}
 If \( T:\wstrs \mapsto \wstrs \) is a pruned \( e \)-deciding  f-tree, \( f \in [T] \) and \( X = \psi_e(f) \) then \( X \) is either arithmetic in \( T \) or \( f \) is arithmetic in \( X \Tplus T \).       
 \end{lemma}
 \begin{proof}

 As every path through \( T \) is \( n \)-generic with respect to \( \forces[T] \) if \( f \forces[T] \psi_e(x) \) then \( x \in \psi_e(f) \). Hence, if \( T \) is totally non-\( e \)-fsplitting then (the characteristic function for) \( X \) is the union of \( \frciseg[T]{e}{\tau} \) for \( \tau \) in the image of \( T \).   If \( T \) is totally \( e \)-fsplitting and \( X = \psi_e(f) \) then \( f \) is the union of \( \tau \) in the image of \( T \) with  \( \frciseg[T]{e}{\tau}  \) compatible with \(  X \).  As forcing for arithmetic sentences is arithmetic the conclusion follows.      
 \end{proof}

 We now prove the key lemma we'll use to establish Theorem \ref{thm:arith-min-w-tree}.  

 \begin{lemma}\label{lem:e-deciding-subtree}
Given \( e \) and  a pruned \( \omega \)-branching  perfect f-tree \( T\!: \wstrs \mapsto \wstrs \) there is a pruned \( \omega \)-branching  perfect \( e \)-deciding  subtree \( \hat{T} \) uniformly arithmetic in \( T \).  
 \end{lemma}
 \begin{proof}
We first observe that given a pruned perfect f-tree \( T \) there is a pruned perfect  subtree \( \tilde{T} \) where every path through \( \tilde{T} \) is \( n \)-generic with respect to \( \forces[T] \).  Moreover, we note that since \( \sigma \forces[T] \psi \implies \sigma \forces[\tilde{T}] \psi \) every path is also \( n \)-generic with respect to \( \tilde{T} \) and if \( T \) was either totally non-\( e \)-fsplitting or totally \( e \)-fsplitting and \( \psi_e \in \sigmazn{n} \union \pizn{n} \)  than \( \tilde{T} \) is \( e \)-deciding.  As we can take \( \tilde{T} \) to be arithmetic in \( T \),   it is enough to demonstrate that there is pruned \( \omega \)-branching  perfect subtree of \( T \),   \( \hat{T} \) arithmetic in \( T \)  that is either totally non-\( e \)-fsplitting or totally \( e \)-fsplitting.          

We consider two cases.

\begin{pfcases}

\case\label{lem:e-deciding-subtree:case:non-splitting}  Suppose some \( \sigma \in \rng T \) isn't extended by any  \( e \)-fsplit \( \tau_0, \tau_1 \in \rng T \).  Let  \( \hat{T} = \subtree{T}{\sigma} \), the subtree of \( T \) above \( \sigma \).  Clearly, \( \hat{T} \) is totally non-\( e \)-fsplitting.

\case Suppose Case \ref{lem:e-deciding-subtree:case:non-splitting} doesn't hold.  We now seek to build a perfect pruned \( \omega \)-branching totally \( e \)-fsplitting subtree \( \hat{T} \) of \( T \).  Absent the need to be \( \omega \) branching we could simply search for an \( e \)-fsplitting pair extending \( \hat{T}(\sigma) \) on \( T \) to define  \( \hat{T}(\sigma\concat[0]), \hat{T}(\sigma\concat[1]) \).  However, if we then later tried to define \( \hat{T}(\sigma\concat[i]) \) we might be unable to choose a value for \( \hat{T}(\sigma\concat[i]) \) which \( e \)-fsplits with \( \hat{T}(\sigma\concat[0]) \).  We could further extend \( T(\sigma\concat[0]) \) to find a splitting but we risk having to do this infinitely often.  Instead, we make sure that when we pick a value for \( \hat{T}(\sigma\concat[0]) \) it \( e \)-fsplits with extensions of infinitely many other strings of the form  \( \hat{T}(\sigma)\concat[i] \).

To define \( \hat{T}(\sigma\concat[1]) \) we essentially repeat the above process and choose one of the infinitely many extensions which \( e \)-fsplits with \( \hat{T}(\sigma\concat[0]) \) that can be extended to \( e \)-fsplit with (extensions of) an infinite subset of those allowed extensions.  The seemingly daunting details spelled out below are just the bookkeeping needed to allow us to repeatedly select a value for \( \hat{T}(\sigma\concat[i]) \) and ensure that we have an infinite set of options that \( e \)-fsplit with \( \hat{T}(\sigma\concat[i']), i' \leq i \) to use for  \( \hat{T}(\sigma\concat[i']), i' > i \) (and to do this for every \( \sigma \)).

We build an f-tree \( V\!: \wstrs \mapsto \wstrs \) and define \( \hat{T} = T \mathbin{\circ} V \).  Our construction of \( V \) will proceed in stages.  At stage \( s \) we will  define \( V(\sigma) \) where \( s = \godelnum{\sigma} \) (remember, that \( \sigma \subfun \tau \implies \godelnum{\sigma} < \godelnum{\tau}  \)).  At stage \( 0 = \godelnum{\estr} \) we begin by defining \( V(\estr) = \estr \). Note that, since \( i < j \implies \godelnum{\sigma\concat[i]} < \godelnum{\sigma\concat[j]} \) if \( \godelnum{\sigma\concat[n]} = s \) then we've already defined \( V(\sigma\concat[m]), m < n \).    

Once we've defined \( V(\sigma) \) we maintain a set \( S^{\sigma}_s \) of extensions of \( \sigma \) representing possible initial segments of \( V(\sigma\concat[i])  \) for \( V(\sigma\concat[i])  \) not yet defined at \( s \).  To ensure that \( \hat{T} \) remains \( \omega \) -branching we ensure that all elements in \( S^{\sigma}_s \) extend incompatible immediate extensions of \( \sigma \).  Unless otherwise specified, \( S^{\sigma}_{s+1} = S^{\sigma}_s \).    

 Suppose that at stage \( s \) we are working to define \( V(\sigma\concat[n]) \), i.e., \( \godelnum{\sigma\concat[n] } = s  \).     Let \( \tau \) be the lexicographically least element of \( S^{\sigma}_s \) and \( U = S^{\sigma}_s  \setminus \set{\tau}{}  \).  Let \( \tau_0, \tau_1 \supfun \tau \) be such that \( T(\tau_0), T(\tau_1) \) \( e \)-fsplit.  Such strings must exist or Case \ref{lem:e-deciding-subtree:case:non-splitting} would have obtained.  Let \( U_i, i \in \set{0,1}{} \) be the set of \( \subfun \) minimal \( \upsilon \) extending some element in \( U \) such that \( T(\upsilon) \)  and \( T(\tau_i) \)  \( e \)-fsplit.  Let \( i \in \set{0,1}{} \) by the least such that \( U_i \) is infinite.  Since the `images' of   \( T(\tau_0)  \) and \( T(\tau_1) \) under \( \psi_e \) disagree such an \( i \) must exist.   

 Set \( V(\sigma\concat[n]) = \tau_i \) and \( S^{\sigma}_{s+1} = U_i \), initialize \( S^{\sigma\concat[n]} \) to \( \set{\tau_i\concat[m]}{m \in \omega} \) and proceed to stage \( s + 1 \).  This completes our definition of \( V \) and clearly defines a pruned f-tree \( \hat{T} = T \circ V \) arithmetically in \( T \).  

 We claim that  \( \hat{T} \) is a totally \( e \)-fsplitting perfect \( \omega \)-branching subtree of \( T \).  Our definition of \( S^{\sigma}_s \) ensures that \( \hat{T} \) is \( \omega \)-branching and that \( \hat{T} \) is perfect. Since whenever we define \( V(\sigma\concat[n])  \) we limit \( T(V(\sigma\concat[n'])), n' > n \) to extensions of strings which   \( e \)-fsplit with \( T(V(\sigma\concat[n])) \) it follows that \( \hat{T} \) is totally \( e \)-fsplitting.

\end{pfcases}

As these cases are exhaustive, this suffices to complete the proof.
      
 \end{proof}

 We can now complete the proof of the theorem.

\begin{thm:arith-min-w-tree}
There is a pruned perfect \( \omega \)-branching  f-tree \( T \Tleq \zeron{\omega} \) such that every \( f \in [T] \)  is of minimal arithmetic degree.  
\end{thm:arith-min-w-tree}

 \begin{proof}
 Iteratively applying Lemma \ref{lem:e-deciding-subtree} would immediately suffice to produce a minimal arithmetic degree.  However, we wish to end up with an \( \omega \)-branching tree of such degrees.  To that end, we set \( T_{0}  \) to be the identity function on \( \wstrs \) and inductively define \( T_{n+1}\restr{n} = T_n\restr{n} \) (i.e. equal when applied to strings of length at most \( n \)) and if \( \lh{\sigma} = n \) and \( T^{\sigma}_{n+1} \) is the \( n \) -deciding subtree of \( \subtree{T_n}{T_n(\sigma)} \) produced by Lemma \ref{lem:e-deciding-subtree} then  \( T_{n+1}(\sigma\concat\tau) = T_n(\sigma) \TreeMul T^{\sigma}_{n+1} \).

 Now let \( T \) be the limit of this process, i.e., \( T(\sigma) = T_{\lh{\sigma}}(\sigma) \).  \( T \) is a perfect pruned \( \omega \)-branching f-tree.  Since we defined \( T_{n+1} \) in a uniform arithmetic fashion from \( T_n \) we have \( T \Tleq \zeron{\omega} \).  Finally, if \( f \in [T] \) and \( X = \psi_e(f) \) and \( \sigma = f\restr{e+1} \)  then \( \TreeMod{f}{\sigma} \in [T^{\sigma}_{e+1}]  \) and thus, as \( T^{\sigma}_{e+1} \) is arithmetic, by Lemma \ref{lem:e-deciding-nomin} either \( X  \) is arithmetic or \( \TreeMod{f}{\sigma} \Tequiv f \) is arithmetic in \( X \).  
 \end{proof}

 \subsection{Fast Growth and Definability}

 In retrospect, perhaps we shouldn't be too surprised by the result in the last section.  After all, there are minimal Turing degrees of hyperimmune degree (e.g. any minimal degree below \( \zeroj \)).  And maybe we don't have to completely give up the idea of using non-domination strength to show that no \( \REA[\omega] \) set can be of minimal arithmetic degree.    While, \textit{surprisingly}, unlike a true \( 1 \)-generic, a weak \( 1 \)-generic can be of minimal Turing degree (a minimal degree below \( \zeroj \) can't be hyperimmune-free and thus computes a  weak \( 1 \)-generic which, by virtue of being non-computable, must be of that very minimal degree), we were still able to use non-domination strength to demonstrate the non-existence of minimal r.e. degrees by identifying a property (weak \( 1 \)-genericity) that enough non-domination strength would less us satisfy (compute) but which couldn't hold of any set of r.e. degree.  Perhaps we could similarly show that every arithmetically escaping function computes a non-arithmetic set with a property that guarantees it's not of \( \REA[\omega] \) degree.

 What might play the role of this property in the arithmetic degrees?  The result in \cite{Andrews2014Degrees} tells us that it can't be \( \omega \)-genericity but the following lemma suggests a different way of generalizing the idea that more non-domination strength should somehow allow us to compute less definable sets.

 \begin{lemma}\label{lem:generic-no-singleton}
If \( X \subset \omega \) is \( n \)-generic with respect to local forcing on some perfect tree (or weakly \( n \)-generic ) then \( X \) isn't a \( \pizn{n} \) singleton.  Similarly, no \( n \)-generic \( f \in \baire \) is a \( \pizn{n} \) function singleton.      
 \end{lemma}  
 \begin{proof}
Suppose that \( X \) is the unique set such that \( X \models \psi \) for some \( \pizn{n} \) formula (with a set constant) \( \psi \).  By \( n \)-genericity  we must have \( X\restr{l} \forces[T] \psi \) for some \( l \).  As \( T \) is perfect there is some \( n \)-generic path  \( Y \supfun X\restr{l}, Y \neq X \) through \( T \).  But, by \( n \)-genericity \( Y \models \psi \) contradicting the fact that \( X \) was the unique solution.

To show the claim holds for weakly \( n \)-generic sets suppose that \( \psi = \forall(x) \Psi(x) \).  Now let \( S = \set{\sigma}{\exists(x)\left( \sigma \forces \lnot \Psi(x) \right)} \).  \( S \) is a \( \sigmazn{n} \) set and if no element in \( S \) extended \( \tau \) then \( \tau \forces^{w} \psi \) and, as above, would contradict the uniqueness of \( X \).  Hence, \( S \) is a \( \sigmazn{n} \) dense set of strings and, as every  weak \( n \) generic is \( n - 1 \)-generic and \( \lnot \Psi \in \pizn{n - 1}  \), if \( X \) is a weak \( n \)-generic meeting \( S \) then \( X \models \lnot \psi \).        

The argument for function singletons proceeds identically.

 \end{proof} 

In this light, we can think of the results from \cite{Andrews2014Degrees} as showing us that any \( \Tzero \)-escaping function computes a set that's not a \( \pizn{1} \) singleton and any \( \zeroj \)-escaping function computes a set that's not a \( \pizn{2} \) singleton\footnote{Since \( \pizn{2} \) singletons are closed under Turing equivalence, we could say not of \( \pizn{2} \) singleton degree.}.  We leave it as an exercise to demonstrate that the techniques in \cite{Andrews2014Degrees} show that sufficient non-domination strength allows us to compute a set that's not a \( \pizn{3} \) singleton.  Thus, a plausible conjecture is that an arithmetically escaping function computes a set that's not an arithmetic singleton (i.e. not a \( \pizn{n} \) singleton for any \( n \)).  If true, this would prove that no \( \REA[\omega] \) set is on minimal arithmetic degree.

\begin{proposition}\label{prop:no-arith-singleton-no-minimal}
    If every arithmetically escaping function \( f \) can arithmetically define a set that's not an arithmetic singleton  then  no \( \pizn{2} \) singleton, and hence no \( \REA[\omega] \) set,  is of minimal arithmetic degree.    
\end{proposition}  
\begin{proof}
Suppose, for contradiction, \( X \) is a  \( \pizn{2} \) singleton of minimal arithmetic degree. By Lemma \ref{lem:uniform-modulus} \( X \) computes an arithmetically escaping \( f \). Let \( Y \Aleq f \) such that \( Y \) isn't an arithmetic  singleton as per the proposition.   If \( Y \) were arithmetic then \( Y \) would be a \( \sigmazn{n} \) set for some \( n \) and thus an arithmetic singleton.  Therefore, we must have \( Y \Aequiv X \).  Thus, for some arithmetic formulas \( \psi, \Psi \) we have  \( \psi(X) = Y \land  \Psi(Y) = X \) and thus \( Y \) is the unique solution of the arithmetic formula which asserts that \( \Psi(Y) \) satisfies the \( \pizn{2} \) formula defining \( X \) and that \( \psi(\Psi(Y)) = Y \).  Contradiction.                
\end{proof}

\section[A Minimal Singleton]{A Arithmetically Minimal \texorpdfstring{ \( \pizn{2} \)}{\textPi⁰₂} Singleton}\label{sec:min-singleton}

We will now prove that our seemingly plausible conjectures (once again) fail and that there is a \( \pizn{2} \) singleton of minimal arithmetic degree.   Our approach to proving this borrows substantially from Harrington's proof of McLaughlin's conjecture \cite{Harrington1976MclaughlinS}.  As it's easier to work with computable trees than \( \pizn{2} \) classes we'll work in \( \wstrs \).  By Lemma \ref{lem:function-to-set} it will be enough to build a computable tree \( T \subset \wstrs \) with a single path \( f \) of minimal arithmetic degree.  However, doing this in a single step would be dauntingly difficult so we instead break up our construction into steps.  

Specifically, the primary task will be to prove the following proposition.

\begin{proposition}\label{prop:tree-build}
Given a (potentially partial)  tree \(  S \subset \wstrs \) computable in \( \jjump{X} \)  there is a (total) computable tree \( T \subset \wstrs \) and an \( \jjump{X} \) computable partial f-tree \( \hat{T} \) such that
    \begin{enumerate}
        \item\label{prop:tree-build:hat-t} \( \rng \hat{T} \subset T \) and \( [\hat{T}] = [T] \) 
        \item\label{prop:tree-build:homeo}  \( \hat{T}(\cdot) \) is a homeomorphism of \( [S] \) with \( [T] \).
        \item\label{prop:tree-build:non-compute}  If \( f \in [T] \) then \( f \nTleq X \).
        \item\label{prop:tree-build:low2}  If \( g \in [S] \) then \( g \Tplus \jjump{X} \Tequiv \jjump{\left(\hT(g) \Tplus X\right)} \Tequiv \hT(g) \Tplus   \jjump{X}  \).
        \item\label{prop:tree-build:min} If \( f \in [T] \) and \( Y \Tleq f \Tplus X  \) then either \( Y \Tleq X \) or \( f \Tleq Y \Tplus \jjump{X} \).

        \item\label{prop:tree-build:len} For all \( \sigma \in \bstrs \),  \( \lh{\hT(\sigma)} \geq \lh{\sigma} \) (when defined).

    \end{enumerate}
    Moreover, this holds with all possible uniformity.  In particular, given a computable functional \( \Upsilon_2 \) we can effectively produce functionals \( \Upsilon, \hat{\Upsilon} \) so that whenever \( \Upsilon_2(\jjump{X}) =S \) then  \( \Upsilon(X)=T  \) and \(  \hat{\Upsilon}(\jjump{X}) = \hat{T} \) with the properties described above.   

\end{proposition}

We will then leverage this proposition to prove the main theorem below by using it repeatedly to pull down a tree \( T_\alpha \Tleq \zeron{\alpha} \) to a homeomorphic image \( T \Tleq \Tzero \).

\begin{theorem}\label{thm:main}
Given a limit ordinal \( \alpha  < \wck \) and a tree \( T_{\alpha} \subset \wstrs, T_\alpha \Tleq \zeron{\alpha} \) there is a computable tree \( T \) and an f-tree \( \Gamma \Tleq \zeron{\alpha} \) such that \( \Gamma \) is a homeomorphism of \( [T_\alpha] \) with \( [T] \) satisfying the following for all \( f \in [T] \) and ordinals \( \beta < \alpha \) 
\begin{itemize}
    \item \( \jumpn{f}{\beta} \Tequiv f \Tplus \zeron{\beta}  \) and, indeed, \( \jumpn{f}{\alpha} \Tequiv f \Tplus \zeron{\alpha}  \)
    \item \( f \nTleq \zeron{\beta} \)
    \item If \( Y \Tleq \jumpn{f}{\beta} \) then either \( Y \Tleq \zeron{\beta} \) or \( f \Tleq Y \Tplus \zeron{\beta +2} \) 
\end{itemize}
Moreover, this holds with all possibility uniformity.
\end{theorem}     

Note that the final point above immediately entails that if there is some \( \beta < \alpha \) such that  \( Y \Tleq \jumpn{f}{\beta} \) then there is some \( \gamma < \alpha \) such that   either \( Y \Tleq \zeron{\gamma} \) or \( f \Tleq \jumpn{Y}{\gamma}  \).  We now catalogue a number of interesting corollaries, such as the existence of a \( \pizn{2} \) singleton of  minimal arithmetic degree.

\begin{corollary}\label{cor:minimal-singleton}
    There is a \( \pizn{2} \) singleton of minimal arithmetic degree. 
\end{corollary}   

\begin{proof}
Taking \( T_\omega = \set{\str{0^n}}{n \in \omega} \)  immediately produces a \( \pizn{1} \) function singleton of minimal arithmetic degree and applying Lemma \ref{lem:function-to-set} transforms this into a \( \pizn{2} \) singleton of minimal arithmetic degree.
\end{proof}

Our work from the previous section isn't wasted as it allows us to prove the following interesting corollary.

\begin{corollary}\label{cor:non-domination-no-avoid-singleton}
There is an arithmetically escaping function \( f \Tleq \zeron{\omega} \) such that every \( X \Aleq f \) is an arithmetic singleton.  
\end{corollary}
\begin{proof}
By Corollary \ref{cor:minimal-singleton} and Lemma \ref{lem:function-to-set} let \( f \) be a \( \pizn{1} \) singleton of minimal arithmetic degree. By Lemma \ref{lem:uniform-modulus} \( f \)  must be arithmetically escaping and by Proposition \ref{prop:no-arith-singleton-no-minimal} every \( Y \Aleq f \) is an arithmetic singleton.    
\end{proof}

We can also derive some interesting consequences about perfect sets of minimal arithmetic degree.

\begin{corollary}\label{cor:class-of-min-singletons}
There is a perfect \( \pizn{2} \) class \( \mathscr{C} \)  all of whose members are of minimal arithmetic degree and which contains elements of arbitrarily large non-domination strength.  
\end{corollary}
By contain elements of arbitrarily large non-domination strength, we mean that for any countable \( C \subset \baire \) there is an \( X \in \mathscr{C} \) and \( f \in \baire \)  such that \( X \Tequiv f \) and \( f \) isn't dominated by any \( g \in C \).   
\begin{proof}
Take \( T_\omega = \wstrs \).  Since \( \Gamma \) is an f-tree and a homeomorphism at each \( \sigma \in T_\omega \) we can pick \( i \) large which ensures that \( \Gamma(\sigma\concat[i]) \) is large at at \( n = \lh{\Gamma^{\omega}_0(\sigma)} \).    Thus, we can apply the same approach as in  in Corollary \ref{cor:hi-arith-min} to show that there is an element of \( [T] \) that avoids domination by any element in \( C \).  By applying Lemma \ref{lem:function-to-set} we get a  \( \pizn{2} \) class whose members are Turing equivalent to the elements in \( [T] \).   


\end{proof}

 Before we move on to providing proofs of Theorem \ref{thm:main} and Proposition \ref{prop:tree-build}, we end this section by asking a few questions.   While the existence of a \( \pizn{2} \) singleton of minimal arithmetic degree  is suggestive, it isn't quite enough to demonstrate that the project of using the properties of fast growing functions to show  \( \REA[\omega] \) sets can't be of minimal degree fails.   This gives rise to the following question.

\begin{question}\label{q:arith-escaping-avoid-w-rea}
If \( f \) is arithmetically escaping, must there be some  \( X \Aleq f \)  where \( X \) isn't  of \( \REA[\omega] \) arithmetic degree. 
\end{question} 

At first glance, one might think that this question stands or falls with the existence of an \( \REA[\omega] \) set of minimal arithmetic degree.  After all, the question must have a negative answer if there is such an \( \REA[\omega] \) set and if the question has a positive answer then no such set can exist.  However, there is still the possibility that the question has a negative answer and no \( \REA[\omega] \) set is of minimal arithmetic degree.  This would require that there is some non-arithmetic \( \REA[\omega] \) set \( A \) such that every \( B \Aleq A \) is arithmetically equivalent to an \( \REA[\omega] \) set but that doesn't seem beyond the realm of possibility.

Next, inspired by the methods in \cite{Harrington1976MclaughlinS}, we ask if this approach offers any utility if extended up through all ordinals below \( \wck \).

\begin{question}\label{q:beyond-wck}
Do results that create a unique path through \( \kleeneO \) via `non-standard' notations offer a means to extend Theorem \ref{thm:main} to prove interesting results about hyperdegrees?  
\end{question}

Such a generalization isn't as simple as merely applying the construction used in Theorem \ref{thm:main} to some non-standard notation.  That can't work as we could then build a \( \pizn{1} \) function singleton that isn't of hyperarithmetic degree which contradicts the fact that if \( f \) is the unique path through a computable tree \( [T] \) then \( f \in \HYP \).  However, the methods in \cite{Harrington1976MclaughlinS} might allow the definition of a perfect \( \pizn{1} \) class of elements all of which satisfy the conclusion of Theorem \ref{thm:main} with respect to all \( \alpha \) in some linearly ordered path through \( \kleeneO \).  


\begin{question}\label{q:arith-escaping}
Is there a perfect \( \pizn{1} \) function class all of whose elements are both arithmetically minimal and arithmetically escaping?  
\end{question}
The difficulty here is that modifying \( T_\omega \) also results in modification to \( T \).     

\begin{question}\label{q:minimal-pairs}
What kind of ability do we have to control the join of pairs of minimal degrees?  Are there \( \pizn{2} \) singletons \( A, B \) of minimal arithmetic degree such that \( A \Tplus B \Aequiv \zeron{\omega} \)   (or even \( \Tequiv \zeron{\beta} \) for arbitrary \( \beta \in \kleeneO \)).  Could we combine this with the idea in Question \ref{q:beyond-wck} to create a perfect \( \pizn{2} \) class whose elements join to compute \( \kleeneO \)?   
\end{question}

\section{Towers of Trees}

Before we get into the weeds of proving Proposition \ref{prop:tree-build} we first show that it suffices to establish the main theorem.  As not all readers may wish to delve into the details involved in manipulating ordinal notations we segregate the results needed to prove the claim for  \( \alpha \) above \( \omega \) into an appendix.  If you are only interested in the case \( \alpha = \omega \) you may simply take \( \copyord{\beta} = \omega \)  and \( \copylen{n} = n \) for all \( n \in \omega \) and  \( \Ofunc{\omega}(n) =  2n \) and replace the following definition of an even notation with that of an even number.

\begin{definition}\label{def:even-notation}
The ordinal notation \( \beta \) is an even notation if \( \beta = \lambda + 2k \) where \( \lambda \) is either \( 0 \) or a limit notation and  \( k \in \omega \).    
\end{definition}

We break our proof of Theorem \ref{thm:main} into a construction of a uniform sequence of trees \( T_\beta \) and a verification that this uniform sequence of trees has the desired properties. 

\begin{proposition}\label{prop:main-construction}
Given a limit notation \( \alpha \) and a tree \( T_{\alpha} \) computable in \( \zeron{\alpha} \) the following hold

\begin{enumerate}
    \item\label{prop:main:tower} For each even notation \( \beta \Oless \alpha \) there is a tree \( T_\beta \) uniformly computable in \( \zeron{\beta} \).
    \item\label{prop:main:homeo} For all \( \gamma \Oless \beta \Oleq \alpha \), there is a uniformly given f-tree  \( \Gamma^{\beta}_\gamma \Tleq \zeron{\beta} \)  that's a homeomorphism of \( [T_\beta] \) with \( [T_\gamma] \).

    \item\label{prop:main:succ-step-subtree} For each even notation \( \beta \Oless \alpha \) there is some \( \copylen{\beta} \in \omega \) such that if   \( \sigma \in T_\beta, \lh{\sigma} = \copylen{\beta} \)  then applying Proposition \ref{prop:tree-build} with \( S= \TreeMod{T_{\beta+2}}{\sigma} \) and \( X = \zeron{\beta} \) produces    \( T= \TreeMod{T_\beta}{\sigma} \)  and \( \hT \) where \(\sigma \concat \hT(\tau) = \Gamma^{\beta+2}_{\beta}(\sigma\concat\tau) \) for all \( \tau \).  


\end{enumerate}

\end{proposition} 
For \ref{prop:main:homeo} we understand the uniformity claim to mean that \( \Gamma^{\beta}_\gamma(\tau) \) is given by \( \Gamma(\beta, \gamma, \zeron{\beta}, \tau) \) for a single computable functional \( \Gamma \).  Similarly, for \ref{prop:main:tower} we understand the uniformity claim to mean that there is a single computable functional such that  \( \Upsilon(\beta, \zeron{\beta}, \cdot) \) gives the characteristic function for \( T_\beta \) for all even notations  \( \beta \Oleq \alpha \),  Moreover, that indexes for both functionals are given by a computable function of \( \alpha \) and an index for \( T_\alpha \).

\subsection{Verifying the Main Theorem}

We now prove a utility lemma which will let us show that Theorem \ref{thm:main} follows from the claim above.

\begin{lemma}\label{lem:g-beta-facts-main}
Given \( \alpha,  T_\beta, \Gamma \) as in Proposition \ref{prop:main-construction}  and \( f \in [T_0] \) for all even notations \( \beta \Oleq \alpha \) 
\begin{enumerate}
    \item\label{lem:g-beta-facts-main:g-beta-exists} There is a unique \( g_\beta \in [T_\beta] \) such that \( \Gamma^{\beta}_{0}(g_\beta) = f \).
    \item\label{lem:g-beta-facts-main:g-beta-compute} \( g_\beta \) is uniformly computable from \( f \Tplus \zeron{\beta} \).
    \item\label{lem:g-beta-facts-main:jump-invert} \( \jumpn{f}{\beta} \Tequiv f \Tplus \zeron{\beta} \Tequiv g_\beta \Tplus \zeron{\beta} \).  Moreover, this equivalence holds uniformly in \( \beta \). 
    \item\label{lem:g-beta-facts-main:no-compute} \( \beta \Oless \alpha \) implies \( g_\beta \nTleq \zeron{\beta} \)  
    \item\label{lem:g-beta-facts-main:min-step} If \( \beta \Oless \alpha \) and \( Y \Tleq g_\beta \Tplus \zeron{\beta} \) then either \( Y \Tleq \zeron{\beta} \) or \( g_\beta \Tleq Y \Tplus \zeron{\beta +2} \)        
\end{enumerate}
\end{lemma}
\begin{proof}
By point \ref{prop:main:homeo} of Proposition \ref{prop:main-construction}  there is a unique \( g_{\beta} \) whose image under  \( \Gamma^{\beta}_0 \) is \( f \).  We can compute \( g_\beta \) in \( \zeron{\beta}  \Tplus f \) as \( \Gamma^{\beta}_0 \Tleq \zeron{\beta} \) and as it's an f-tree it sends incompatible elements in the domain to incompatible elements in the range.  This verifies points \ref{lem:g-beta-facts-main:g-beta-exists} and \ref{lem:g-beta-facts-main:g-beta-compute}.

For point \ref{lem:g-beta-facts-main:no-compute} let \( \sigma = g_\beta\restr{\copylen{\beta}} \) and note that by \ref{prop:main:succ-step-subtree} of Proposition \ref{prop:main-construction} for \( \beta \Oless \alpha \)  the tree \( \TreeMod{T_\beta}{\sigma} \) is the result of Proposition \ref{prop:tree-build} applied to \( \TreeMod{T_{\beta+2}}{\sigma} \) and  by \ref{prop:tree-build:non-compute} of Proposition \ref{prop:tree-build} we know that \( g_{\beta} \nTleq \zeron{\beta} \).  By part \ref{prop:tree-build:min} of Proposition \ref{prop:tree-build} we have that \ref{lem:g-beta-facts-main:min-step} holds.  This leaves us only \ref{lem:g-beta-facts-main:jump-invert} to verify.

Suppose that \ref{lem:g-beta-facts-main:jump-invert} holds for all notations \( \gamma \Oless \beta \) (where we regard the claim as trivially true for odd notations).  We prove the claim holds for \( \beta \).  First, suppose that \( \beta = \gamma + 2 \).

Now let \( \sigma = g_\gamma\restr{\copylen{\gamma}} \).  By \ref{prop:main:succ-step-subtree} of Proposition \ref{prop:main-construction} the tree \( \TreeMod{T_\gamma}{\sigma} \) is the result of Proposition \ref{prop:tree-build} applied to \( \TreeMod{T_{\gamma+2}}{\sigma} \).  Thus, by point \ref{prop:tree-build:low2} of Proposition \ref{prop:tree-build} we have \( \jjump{\left(g_\gamma \Tplus \zeron{\gamma}\right)} \Tequiv g_\gamma \Tplus \zeron{\gamma +2} \Tequiv g_{\gamma+2} \Tplus \zeron{\gamma+2} \).  The inductive assumption transforms this equivalence into the one asserted in \ref{lem:g-beta-facts-main:jump-invert}. 

Now, suppose that \( \beta \) is a limit notation.  Note that as Proposition \ref{prop:main-construction} guarantees that \( T_\gamma \) is uniformly computable from \( \zeron{\gamma} \) (for \( \gamma \) an even notation) and Proposition \ref{prop:tree-build} holds uniformly therefore claim  \ref{lem:g-beta-facts-main:jump-invert} holds uniformly below \( \beta \).  But \( \zeron{\beta} \) can uniformly compute \( \zeron{\gamma} \)  and thus \( f \Tplus \zeron{\beta} \) can compute \( \jumpn{f}{\gamma} \) for notations \( \gamma \Oless \beta \).  Hence \( f \Tplus \zeron{\beta} \) can compute \( \jumpn{f}{\beta}  \) which is automatically an equivalence.  Using \( \Gamma^{\beta}_0 \Tleq \zeron{\beta} \) to translate between \( g_\beta \) and \( f \)  we see \( f \Tplus \zeron{\beta} \Tequiv g \Tplus \zeron{\beta} \) completing the verification.       
\end{proof}

We can now show that Theorem \ref{thm:main} follows from Proposition \ref{prop:main-construction}.

\begin{proof}
Take \( \alpha' \) to be a notation for the desired ordinal in Theorem \ref{thm:main} and apply  Proposition \ref{prop:main-construction} to get the notation \( \alpha \), the sequence \( T_\beta \) and \( \Gamma \).   We now apply Lemma \ref{lem:g-beta-facts-main}.  The equivalence \(  \jumpn{f}{\beta} \Tequiv f \Tplus \zeron{\beta} \Tequiv g_\beta \Tplus \zeron{\beta} \) from part \ref{lem:g-beta-facts-main:jump-invert} of Lemma \ref{lem:g-beta-facts-main} gives the first claim in Theorem \ref{thm:main} on it's own.  Applying it to part \ref{lem:g-beta-facts-main:min-step}  of Lemma \ref{lem:g-beta-facts-main} is enough to give the final claim in Theorem \ref{thm:main}.  Finally, if \( f \Tleq \zeron{\beta} \) then that equivalence implies that \( g_\beta \Tplus \zeron{\beta} \Tleq \zeron{\beta} \) contradicting part \ref{lem:g-beta-facts-main:no-compute} of Lemma \ref{lem:g-beta-facts-main}.    
\end{proof}

\subsection{Constructing the Tower}

We now prove Proposition \ref{prop:main-construction} follows from Proposition \ref{prop:tree-build}.  To deal with notations for ordinals above \( \omega \) we use some results proved in the appendix. We start by defining \( \copylen{\beta}, \copyord{\beta} \) via Definition \ref{def:copylen-copyord}.  As mentioned above, readers only interested in the case \( \alpha = \omega \) can skip the material on ordinal notations in the appendix.  They will, however, still need to know that we can replace \( T_\omega \) with a tree \( T'_\omega \) with \( [T'_\omega] = [T_\omega] \)  in which the membership of \( \sigma \in  T'_\omega \) for \( \lh{\sigma} = 2n \)  can be computed uniformly in \( \zeron{2n} \) (a special case of the general result proved in Lemma \ref{lem:tree-uniformization}).

We construct a computable functional \( \Upsilon \)  such that for all even notations \( \beta \) with \( \beta \Oleq \alpha \), \( \Upsilon(\beta, \zeron{\beta}, \cdot) \) gives the characteristic function for \( T_\beta \).  Regarding functionals as r.e-sets of axioms we can define a functional \( \Upsilon \) assuming that we already have access to some partial functional \( \hat{\Upsilon} \) and then use the recursion theorem to yield a single functional satisfying \( \Upsilon = \tilde{\Upsilon} \).  Formally speaking, \( \Upsilon \) and \( \tilde{\Upsilon} \) are defined to be sets of axioms however, we will only specify those sets implicitly by  instead defining the trees \( T_\beta \) in terms of the trees \( \tilde{T}_\beta \) where we understand that \( \tilde{T}_\beta \) represents the set defined by \( \tilde{\Upsilon}(\beta, X, \cdot) \) on the guess that \( X = \zeron{\beta} \).     


Specifically, we define \( T_\alpha \) to be \( T_\alpha \).  That is, regardless of the behaviour of \( \tilde{\Upsilon} \), \( \Upsilon(\alpha, X, \cdot) \) gives the computation that yields the characteristic function for \( T_\alpha \) when \( X = \zeron{\alpha}  \).  Given \( \beta \) an even notation with \( \beta \Oless \alpha \) we define the tree \( T_\beta \) as follows (once \( \Upsilon \) has verified the \( \sigmazn{1} \) fact that \( \beta \Oless \alpha \) is an even notation).

Let \( \lambda = \copyord{\beta} \).  If \( \lh{\sigma} \leq \copylen{\beta} \) then \( \sigma \in T_\beta \) iff \( \sigma \in \tilde{T}'_\lambda \) where if \( \lambda \) is a limit notation, \( \tilde{T}'_\lambda \) is the result of applying Lemma \ref{lem:tree-uniformization} to \( \tilde{T}_\lambda \) so that  \( \tilde{T}'_\lambda\restr{\copylen{\beta}} \) has membership uniformly computable in \( \zeron{\beta} \) and \( [\tilde{T}'_\lambda] \) = \( [\tilde{T}_\lambda] \).  If \( \lambda \) is a successor notation than \( \tilde{T}'_\lambda = \tilde{T}_\lambda \).  

For \( \sigma' \) with \( \lh{\sigma'} > \copylen{\beta} \) let \( \sigma' = \sigma \concat \tau \) where \( \lh{\sigma} = \copylen{\beta} \).  Set \( \sigma' \nin T_\beta \) if \( \sigma \nin \tilde{T}'_\lambda \).  If \( \sigma \in \tilde{T}'_\lambda \) then apply Proposition \ref{prop:tree-build}  to \( \tilde{T}^{\sigma}_{\beta+2} \eqdef \TreeMod{\tilde{T}_{\beta +2}}{\sigma} \) to yield some \( T^{\sigma}_\beta \) and place \( \sigma' \in T_\beta \) just if \( \tau \in T^{\sigma}_\beta  \).  Note that, we can enforce the fact that \( \tilde{T}^{\sigma}_{\beta+2} \) is always a tree by only allowing strings into this set when all their predecessors have been seen to be in the set.   

We now define \( \Upsilon \) to be the fixed point rendering \( \tilde{\Upsilon} = \Upsilon \).  Thus, for an even notation \( \beta \Oleq \alpha \) we define  \( T_\beta  = \Upsilon(\beta, \zeron{\beta}) \).  Note that, by part \ref{prop:trees-copylen:between} of  Proposition \ref{prop:trees-copylen} we can be sure that our fixed point will satisfy \( T_\beta\restr{\copylen{\beta}} \subset T_{\beta+2}\restr{\copylen{\beta}}  \) so that \( \tilde{T}^{\sigma}_{\beta+2} \) will be defined for \( \sigma \in T_{\beta}, \lh{\sigma} = \copylen{\beta} \). 

We define the functional \( \Gamma \) in a similar fashion assuming we have some \( \tilde{\Gamma} \) and applying the recursion theorem.  Note that, the argument above produces two indexes \( i \) and \( q(i) \) for the functional \( \Upsilon \) where \( q \) is the computable functional implicitly defined by the construction above and \( i \) is the index given by the recursion theorem applied to \( q \).   While \( i \) and \( q(i) \) result in the same set of axioms, we can't guarantee that the application of Proposition \ref{prop:tree-build}  produces the same \( T, \hT \) for different indexes for the same tree \( S \).  To handle this issue, we use   \( \tilde{T}_{\beta + 2} \) to indicate the definition in terms of \( i \) and apply Proposition \ref{prop:tree-build} to \( \TreeMod{\tilde{T}_{\beta +2}}{\sigma} \) not   \( \TreeMod{T_{\beta +2}}{\sigma} \) to ensure we get the f-tree \( \hT \)  associated with the construction of  \( \TreeMod{T_{\beta}}{\sigma} \).

With this in mind, we define \( \Gamma^{\beta+2}_\beta \) for \( \beta \Oleq \alpha \) an even notation as follows.
If \( \lh{\sigma} \leq \copylen{\beta} \) then \( \Gamma^{\beta+2}_\beta(\sigma) = \sigma \) if \( \sigma \in \tilde{T}_{\beta + 2}  \) and undefined otherwise.  If \( \lh{\sigma'} > \copylen{\beta}  \) then let \( \sigma = \sigma'\restr{\copylen{\beta}} \) and if \( \sigma \nin \tilde{T}_{\beta + 2}  \) then \( \Gamma^{\beta+2}_\beta(\sigma')\diverge \).  Otherwise,  we define \(  \Gamma^{\beta+2}_\beta(\sigma') = \sigma \concat \hT(\tau) \) where , \( \sigma' = \sigma \concat \tau \)  and \( \hT \) is the f-tree produced by Proposition \ref{prop:tree-build} as applied to \( \TreeMod{\tilde{T}_{\beta +2}}{\sigma} \).  For \( \beta' \Oless \beta \) we define \( \Gamma^{\beta+2}_{\beta'} \) to be \( \Gamma^{\beta+2}_{\beta} \circ \tilde{\Gamma}^{\beta}_{\beta'} \).    

Finally, for \( \lambda \) a limit and \( \beta \Oless \lambda \Oleq \alpha \), \( \beta \) an even notation we define \( \Gamma^{\lambda}_{\beta}(\sigma) \) as follows.  Let \( \beta_n \) be the sequence of notations given by part \ref{prop:trees-copylen:lim} of Proposition \ref{prop:trees-copylen}.  Let \( m \) be the least with \( \lh{\sigma} \leq m \) and \( \beta \Oleq \beta_m \) and define \( \Gamma^{\lambda}_{\beta}(\sigma) \) to be equal to \( \tilde{\Gamma}^{\beta_m}_{\beta}(\sigma) \).  Finally, we define \( \Gamma \) via the recursion theorem so that \( \Gamma = \tilde{\Gamma} \).    

To verify this construction produces the desired result, we suppose that \( \beta, \gamma \) is the lexicographically least pair of even notations \( \gamma \Oless \beta \Oleq \alpha \)  such that \( \Gamma^{\beta}_\gamma \) isn't a f-tree that's a homeomorphism of \( [T_\beta] \) with \( [T_\gamma] \) and then observe that if \( \beta = \beta' + 2 \) then since \( \Gamma^{\beta'+2}_{\beta'} \) is a homeomorphism we get a contradiction.  Similarly, if \( \lambda = \beta \) is a limit then by Proposition \ref{prop:trees-copylen}, \ref{prop:tree-build:len} of Proposition \ref{prop:tree-build} and the the inductive assumption regarding  \( \Gamma^{\beta_n}_\beta \)  we can derive that \( \Gamma^{\lambda}_\beta \) is an f-tree that's a homeomorphism of \( [T_\lambda] \) with it's image contained in \( [T_\beta] \).  Now suppose that some \( f \in [T_\beta] \) isn't equal to \( \Gamma^{\lambda}_\beta(g) \) for any \( g \in [T_\lambda] \).  Then, for some \( n \) we have that \( f \supfun \Gamma^{\beta_n}_\beta(\sigma) \) for  no \( \sigma \in T'_\lambda \) with \( \lh{\sigma} = \copylen{\beta_n} \) and thus by the inductive assumption \( f \nin [T_\beta] \).  Thus, the claim holds for \( \Gamma^{\lambda}_\beta \) as well.  This suffices to prove Proposition \ref{prop:main-construction}.

\section{Minimality and Double Jump Inversion}\label{sec:min-and-double-jump-inversion}

In this section, we finally present the proof of Proposition \ref{prop:tree-build}. However, before we do this we answer an obvious question raised by our construction.  Why use double jump inversion and not single jump inversion as Harrington did in \cite{Harrington1976MclaughlinS}.  The answer is that it's not possible.  We can't satisfy the minimality style requirements while also using \( \omega \)-branching to encode the copy of \( T_{\beta + 1} \).   Specifically, we now show that no \( \zeroj \) \( \omega \)-branching tree \( T \) lets us achieve the kind of minimality required by part \ref{prop:tree-build:min} of Proposition \ref{prop:tree-build}.

\begin{proposition}\label{prop:no-do-one-jump}
Given a perfect weakly \( \omega \)-branching pruned f-tree \( T \Tleq \zeroj  \) one can uniformly construct a computable functional  \( \recfnl{}{}{} \) such that \( e \)-splitting pairs in \( T \)  occur above every node in \( T \) and for every \( \tau \in \rng T \)  there   are paths \( f \neq g \) extending \( \tau \)  through \( T \) with \( \recfnl{e}{f}{} = \recfnl{e}{g}{} \).    
\end{proposition}

Indeed, the result is actually slightly stronger in that we show that even if \( T \) is unpruned we can start building such \( f, g \) above every node in \( T \) and always extend to preserve agreement under \( \recfnl{e}{}{} \) until we hit a terminal node.  Even without this improvement, this rules out the possibility of building \( \hT \) to be computable in \( \jump{X} \).  While we aren't guaranteed that \( \hT \) itself is pruned it will be pruned whenever \( S \) is pruned.  Thus, the above result rules out the possibility of performing the construction using only single jump inversion.   Since the proof of this proposition takes some work and would interrupt the flow of the paper we relegate it to Appendix \ref{app:min-double-jump}.

 \subsection{Machinery} The construction will build \( T \) via a stagewise approximation \( T_s \) with \( T_{s+1} \supset T_s \) and \( \sigma \in T \) iff  \( \sigma \in T_{\godelnum{\sigma}} \).   We will also maintain a stagewise approximation \( \hT[s] \) to \( \hT \).  We will organize the construction of  by drawing on ideas from  \( \pizn{2} \) constructions in the r.e. sets.  In particularly, we will arrange modules on a priority tree, denoted \( \pTree \subset \wstrs \), but the differences between constructing a set and a tree require we make some modifications.

Usually, in a \( \pizn{2} \) tree construction we assign modules to elements of \( \wstrs \), denoting an arbitrary module assigned to \( \alpha \)  by  \( \module{M}{\alpha} \) (for a module of type \( R \),  \( \module{R}{\alpha} \)) and use only the outcome of that module at stage \( s \) to determine where next to visit at stage \( s \).  However, rather than working on meeting a requirement globally for the entire tree \( T \) we will assign modules to work on meeting a requirement for the path extending \( \hT(\sigma) \) for some \( \sigma \).  Thus, rather than working on \( \pTree \subset \wstrs \) we work on \( \pTree \subset \wstrs \cross \wstrs \).  More specifically, we identify modules in our construction with pairs \( (\alpha, \sigma) \)  where \( \alpha \) is the string built up out of the outcomes of prior modules and \( \sigma \) indicates the element in \( \dom \hT \) we are working to define.   The idea is that at certain points in our construction, rather than following a single outcome of a module, we will simultaneously work both above (our approximation to) \( \hT(\sigma\concat[m]) \) and above/to define  \( \hT(\sigma\concat[m']) \). 

When we specify a module \( \module{M}{\xi} \)  we also specify a set  \( \Msucc(\xi) \) of pairs of potential successors \( (o, \delta) \)    where \( o \) is a potential outcome of the module (identified with elements in \( \omega \) but written more suggestively) and \( \delta \in \wstrs \).  As usual, at each stage any module we visit will have a single outcome \( o \) but as there may be multiple pairs \( (o, \delta) \in \Msucc(\xi) \) there may be multiple immediate successors of this module which both get visited at this stage.   In our construction most modules will only allow  \( \delta = \estr \) but some modules will have successors both of the form \( (o, \estr) \) and \( (o, \str{n}) \) (for some particular value \( n \)).  Thus, we will only every we working on finitely many modules simultaneously at any stage.    With this in mind, we define our priority tree \( \pTree \) inductively as follows.  In what follows, keep in mind that  \( \decode{\pair{a}{b}}{0} = a  \) and  \( \decode{\pair{a}{b}}{1} = b  \).

\begin{definition}\label{def:enchanced-ptree}
\( \pTree \) is the smallest set of pairs \( \pair{\alpha}{ \sigma} \) closed under the following conditions
\begin{itemize}
    \item \( \pair{\estr}{\estr} \in \pTree \)
    \item If \( \xi= \pair{\alpha}{ \sigma}  \in \pTree \land   (o, \delta) \in  \Msucc(\xi) \) then \( \nu = \pair{\alpha\concat[o]}{\sigma\concat \delta} \in \pTree \).  In this case we write \( \nu^{-} = \xi \) and call \( \xi \) the predecessor of \( \nu \).  
\end{itemize}

We define \( \lh{\xi} \in \pTree = \lh{\decode{\xi}{0}} \) and say \( \xi \subfun \xi' \) if \( \subfun  \) holds on both components (i.e.    \( \decode{\xi}{i} \subfun \decode{\xi'}{i}, i \in \set{0,1}{} \)).  Finally, \( \xi^{-} \) is defined to be the unique \( \subfun \)  maximal element in \( \pTree \) with  \( \xi^{-} \subfun \xi \).

\end{definition}

The careful reader might note the possibility that \( \xi^{-} \) could fail to be unique if we aren't careful.  To avoid this, we assume that  the outcomes of \( \xi = \pair{\alpha}{\sigma} \) are modified to be of the form \( \pair{o}{\sigma} \).  This ensures that \( \xi^{-} \) is uniquely defined and \( \subfun \) is always a linear order when restricted to the predecessors of \( \xi \) on \( \pTree \).  As it won't cause any confusion, we will assume this happens in the background and will present our outcomes untransformed.

We now define what it means for a node on this tree to be to the left of another node (we retain the terminology `left of' even though it's not longer visually accurate) and extend this to a set of nodes as follows.

\begin{definition}\label{def:leftof}
We define \( \xi \leftof \xi' \) on \( \pTree \)  just if there are \( \nu \subfun \xi, \nu' \subfun \xi \) with \( \nu^{-} = {\nu'}^{-} \), \( \decode{\nu}{0} = \alpha\concat[o], \decode{\nu'}{0} = \alpha\concat[o'] \) and \( o < o' \).  We extend this relation to sets by setting \( Q \leftof \xi \)  (read left of) for   \( Q \subset \pTree, \xi \in \pTree \) just if \( Q \) contains an element \( \nu \leftof \xi \).     
\end{definition}

Our truepath, and its approximations, will no longer be single paths but sets of nodes.  Informally speaking, we define \( \tpath[s] \) to be the set of nodes visited at stage \( s \) following the rules described above but not visiting any extensions of a node \( \xi \) being visited for the first time at stage \( s \).  Formally speaking, we give the following definition.

\begin{definition}\label{def:truepath}
We define \( \tpath[s] \) as the largest set satisfying the following closure conditions
\begin{itemize}
    \item \( \estr \in \tpath[s] \) 
    \item If \( \xi \in \pTree, \xi \in \tpath[s] \), \( s_\xi > 0 \)  and at stage \( s \), \( o  \) is the outcome of \( \xi \) at stage \( s \)  and \( \nu \in \pTree, \nu^{-} = \xi \)  with \( \decode{\nu}{0} = \decode{\xi}{0}\concat[o] \) then \( \nu \in \pTree \).  
\end{itemize} 
Where \( s_\xi \) is defined to be \( \card{\set{t}{ \xi \in \tpath[t] \land t < s}} \).

We define \( \xi \in \tpath \) iff \( \exists(s)\forall(s' > s)\left(\lnot \tpath[s] \leftof \xi \right) \land \existsinf(s)\left( \xi \in \tpath[s]\right) \)  
\end{definition}

Note that, out priority construction will never reinitialize any nodes.  That is, our construction will satisfy the following condition.

\begin{condition}\label{cond:no-reinit}
If \( s' > s  \) and  \( \xi \in \tpath[s],  \tpath[s'] \leftof \xi \) then for all \( t \geq s' \) \( \xi \nin \tpath[t] \).    
\end{condition}

With the action of the priority tree defined we need to specify how the modules are able to control the construction.  As described above, modules will directly enumerate elements into \( T = \Union_{s \in \omega} T_s \) with a deadline of stage \( s = \godelnum{\sigma} \) to place \( \sigma \) into \( T \).   We use a bit more machinery to specify our approximation to \( \hT \).  Each module \( \module{M}{\xi} \) receives a string \( \Mout{\xi} \) and specifies a string \( \Mout{\nu} \) for each successor to \( \xi \).  If \( \xi = \pair{\alpha}{\sigma} \) then we understand the module \( \xi \) to be executing on the guess that \(  \Mout{\xi} \subfun \hT(\sigma) \).     

This is sufficient for modules that only need to manipulate a single path but some modules will need to manipulate the collection of potential branches of  \( \hT(\sigma) \).  To this end, some modules will also define an infinite set  \( \Mexts{\xi}  \) of branches with the \( n \)-th element (ordered lexicographically) indicated by \( \Mext[n]{\xi} \).  In our construction, we will ensure that our definition of \( \Mout{\xi} \) and \( \Mexts{\xi} \) satisfy the following condition.

\begin{condition}\label{cond:output}
For each \( \xi \in \pTree \) 
\begin{enumerate}
    \item If \( \xi \in \tpath[s] \land s_\xi = 0  \) then \( \module{M}{\xi^{-}} \) must set \( \Mout{\xi} \) during stage \( s \) and ensure \( \Mout{\xi} \in T_{s+1} \) . 
    \item \( \nu \supfuneq \xi \implies \Mout{\xi} \supfuneq \Mout{\nu} \).

\end{enumerate}
If  \( \xi \in \pTree \) and \( \module{M}{\xi^{-}} \) defines \( \Mexts{\xi} \) then
\begin{enumerate}
    \item \( \Mext[n]{\xi} \) enumerates  \( \Mexts{\xi} \) with \( \Mext[n]{\xi}  \supfun \Mout{\xi}\concat[k_n]   \) with \( n \mapsto k_n \) monotonic, injective function of \( n \).
    \item If   \( \xi \in \tpath[s] \land s_\xi = n  \) then \( \module{M}{\xi^{-}} \) must set \( \Mext[n]{\xi} \) by the end of stage \( s \) and ensure \( \Mext[n]{\xi} \in T_{s+1} \).
\end{enumerate}
\end{condition}

These are mostly straightforward demands that what the module at \( \xi \) does is compatible with what \( \xi^{-} \) does and defines it's output promptly.  However, a few points deserve mentioning.  The requirement that \( \Mout{\xi} \in T_{s+1} \) will ensure that \( \rng \hT \subset T \).  The final condition will enable multiple modules who all want to ensure the leftmost branch extending  \( \hT(\sigma) \) has some property to cooperate.  Without this condition, a module that only ensured \( \Mext[0]{\xi} \)   has some property might find their work erased by the next module leaving all extensions of \( \Mext[0]{\xi} \) out of the set of branches it specifies.  

We also impose the following condition on the construction to (help) ensure that if \( \xi \in \tpath \) then the modules above \( \xi \) get to control whether \( \Mout{\xi} \) extends to a path through \( T \).

\begin{condition}\label{cond:mod-seperation}
If \( \Mout{\xi} = \tau \) and the module \( \module{M}{\xi} \) enumerates \( \sigma \) into \( T \) then \( \sigma \supfun \tau \).   
\end{condition}

\subsection{Requirements}

With an understanding of how our \( \pTree \) operates we are now in a position to present the requirements our construction will meet and arrange the modules we will use to meet them on the tree.  Recall that we seek to prove the following result.

\begin{prop:tree-build}

 Given a (potentially partial)  tree \(  S \subset \wstrs \) computable in \( \jjump{X} \)  there is a (total) computable tree \( T \subset \wstrs \) and an \( \jjump{X} \) computable partial f-tree \( \hat{T} \) such that
    \begin{enumerate}
        \item \( \rng \hat{T} \subset T \) and \( [\hat{T}] = [T] \) 
        \item \( \hat{T}(\cdot) \) is a homeomorphism of \( [S] \) with \( [T] \).
        \item If \( f \in [T] \) then \( f \nTleq X \).
        \item If \( g \in [S] \) then \( g \Tplus \jjump{X} \Tequiv \jjump{\left(\hT(g) \Tplus X\right)} \Tequiv \hT(g) \Tplus   \jjump{X}  \).
        \item If \( f \in [T] \) and \( Y \Tleq f \Tplus X  \) then either \( Y \Tleq X \) or \( f \Tleq Y \Tplus \jjump{X} \).

        \item For all \( \sigma \in \bstrs \),  \( \lh{\hT(\sigma)} \geq \lh{\sigma} \) (when defined).   
    \end{enumerate}
    Moreover, this holds with all possible uniformity.  In particular, given a computable functional \( \Upsilon_2 \) we can effectively produce functionals \( \Upsilon, \hat{\Upsilon} \) so that whenever \( \Upsilon_2(\jjump{X}) =S \) then  \( \Upsilon(X)=T  \) and \(  \hat{\Upsilon}(\jjump{X}) = \hat{T} \) with the properties described above.   

\end{prop:tree-build}

We work to meet the following requirements during the construction.  We state the requirements in unrelativized form.  Unlike requirements in the construction of a single set, we work to ensure that the requirement of the form \( \req*{R}{e} \) is satisfied for all \( e \) and all \( \sigma \in \pruneTree{S}  \) with \( \lh{\sigma} = e \).     For the purposes of stating the requirements, we use \( \hT-(\sigma) \) to denote an extension of  \( \hT(\sigma^{-}) \) that would be extended by \( \hT(\sigma) \) if the later were defined.  We won't actually define this function but just use it in the requirements to stand in for some string to be defined later\footnote{We could define \( \hT-(\sigma) \) to be \( \Mout{\xi} \) for the unique \( \xi \) along the truepath of the form \( (\alpha, \sigma) \) such that \( \xi \) implements \( \module{H}{\sigma} \) but, as it's unnecessary for the proof, we feel doing so would unnecessarily multiply notation.}.

\begin{requirements}
\require{P}{e}  \hT(\sigma) \incompat \recfnl{e}{X}{} \lor \recfnl{e}{X}{}\diverge \\[.5em]
\require{L}{e}  \forall({f \in [T], f \supfun \hT(\sigma)})\left(\recfnl{e}{f \Tplus X}{}\diverge\right) \lor \forall({f \in [T], f \supfun \hT(\sigma)})\left(\recfnl{e}{f \Tplus X}{}\conv\right) \\[.5em]
\require{H}{\sigma}  \begin{aligned} 
                            S(\sigma)\conv = 0 &\implies \hT(\sigma)\conv \land \forall(m)\left(\hT(\sigma)\concat[m] \nin T \right) \\ 
                            S(\sigma)\conv = 1 &\implies \hT(\sigma)\conv \in \pruneTree{T} \iff \sigma \in \pruneTree{S}  \\
                            S(\sigma)\diverge &\implies \hT(\sigma)\diverge \land \hT-(\sigma) \nin \pruneTree{T}
                        \end{aligned} \\[.5em]
\require{S}[n]{e}  \begin{gathered}
                        \subtree{T}{\hT(\sigma\concat[n])} \text{ is totally non-\( e \)-splitting } \lor m > n \implies \\
                        \hT-(\sigma\concat[n]) \text{ and } \hT-(\sigma\concat[m]) e-\text{split}
                    \end{gathered}
\end{requirements}

Note that \( \recfnl{e}{f}{}\conv \) means  \( \forall(n)\left(\recfnl{e}{f}{n}\conv\!\right) \) (and similarly for \( \recfnl{e}{f}{}\diverge \)) and that, when we speak of \( e \)-splittings we mean the notion relativized to \( X \).  The statement of \( \req{H}{\sigma} \) is a bit odd in the case where \( \sigma \in S \) since in that we do nothing except don't try to stop \( \hT(\sigma) \) from potentially extending to a full path should \( \sigma \) extend to a full path through \( S \).

Each requirement gets it's own module to assist in meeting it, however, some modules get helper modules.  For instance, we break up meeting the requirement \( \req{H}{\sigma} \) into a module \( \module{H}[+]{\sigma^{-}} \) responsible for creating an \( \omega \)-branching above \( \hT(\sigma^{-}) \) and a module \( \module{H}{\sigma} \) responsible for ensuring \( \hT(\sigma) \) doesn't extend to a path through \( T \)  if  \( \sigma \nin S \).   Similarly, we supplement \( \module{L}{e} \) with submodules \( \module{L}[n]{e} \)  responsible for checking if we can extend \( \recfnl{e}{f}{} \) to converge on \( \recfnl{e}{f}{n} \).      Finally, we use the module \( \module{S}[n]{-1} \) as a helper to split off those modules who will work above \( \hT(\sigma\concat[n]) \) from those modules working to define  \( \hT(\sigma\concat[m]), m > n \).

\begin{definition}\label{def:modules-on-ptree}
Modules are be assigned to nodes on \( \pTree \) as follows 
\begin{enumerate}
    \item If  \( \xi = (\estr, \estr) \) then \( \xi \) implements \( \module{H}[+]{\estr} \).  Also, if \( \xi^{-} \) implements some module \( \module{H}{\sigma} \) and has an outcome guessing \( \sigma \in S \) then \( \xi \) implements  \( \module{H}[+]{\sigma} \).  

    \item If \( \xi^{-} \) implements  \( \module{H}[+]{\sigma} \) with \( \lh{\sigma} = e \) then \( \xi \) implements  \( \module{S}[0]{e} \).

    \item If  \( \xi^{-} \) implements \( \module{S}[n]{e} \) with \( e \geq 0 \) then \( \xi \) implements   \( \module{S}[n]{e - 1} \).

    \item If \( \xi^{-} \) implements \( \module{S}[n]{-1} \) and \(\decode{\xi}{1} = \decode{\xi^{-}}{1} = \sigma  \) then \( \xi \) implements \( \module{S}[n+1]{e} \) where \( \lh{\sigma} = e \).

    \item If \( \xi^{-} \) implements \( \module{S}[n]{-1} \) and \( \decode{\xi}{1} \neq \decode{\xi^{-}}{1}  \)  then \( \xi \) implements \( \module{P}{e} \) for the least \( e \) such that no module of this form is assigned to any \( \nu \subfunneq \xi \).

    \item If \( \xi^{-} \) implements \( \module{P}{i} \) then \( \xi \) implements \( \module{L}{e} \) for the least \( e \) such that no \( \nu \subfunneq \xi \) implements \( \module{L}{e} \). 

    \item If \( \xi^{-} \) implements a module of the form \( \module{L}{e} \) or \( \module{L}[n]{e} \) then let \( i < e \) (if it exists) be the largest value such that \( \xi \) extends the outcome \( \godelnum{\conv} \) of  \( \module{L}{i} \) and \( m \) be the least such that no predecessor of \( \xi \) implements  \( \module{L}[m]{i} \) then \( \xi \) implements  \( \module{L}[m]{i} \). 

    \item If \( \xi^{-} \) implements a module of the form  \( \module{L}{e} \) or \( \module{L}[n]{e} \) and either \( e = 0 \) or for all \( i < e \) \( \xi^{-} \)  doesn't extend the \( \godelnum{\conv}\) outcome of \( \module{L}{i} \) 
    and  \( \sigma = \decode{\xi}{1} \) then \( \xi \)  implements the module \( \module{H}{\sigma} \)

\end{enumerate}     

\end{definition}

Before we get into any further details, we give a high level overview of how this is all supposed to work.  If we suppose that we've just defined   \( \hT(\sigma) \) and wish to define \( \hT(\sigma\concat[m]) \) we start with the module \( \module{H}[+]{\sigma} \) which will specify a bunch of immediate extensions of \( \hT(\sigma) \) (placing them in \( T \)).  We start by executing \( \module{S}[0]{e}, \lh{\sigma} = e \) then \( \module{S}[0]{e -1 } \) and so forth all of which work to ensure the leftmost potential extension of \( \hT(\sigma) \) \( e \)-split or \( e -1 \) split or so on with all the remaining potential extensions.  When we finally get to the module \( \module{S}[0]{-1} \) it specifies that \( \hT(\sigma\concat[0]) \) extends the leftmost branch as extended by all the modules \( \module{S}[0]{e'}, e' \leq e \) and the construction now splits into one part which works on the next module \( \module{P}{e} \) along that path specified as an initial segment of \( \hT(\sigma\concat[0]) \) and another part where \( \module{S}[1]{e} \) starts working to define the node which will be extended by \( \hT(\sigma\concat[1]) \).  After the module of the form \( \module{P}{e} \) we work on the next module of the form \( \module{L}{e} \) and then, if we are above the total outcome any \( \module{L}{e'}, e' \leq e \) we implement the next helper module of the form \( \module{L}[m]{e'} \).  Finally, after those modules, comes the module \( \module{H}{\sigma\concat[m]} \) (assuming we took the path working on \( \hT(\sigma\concat[m]) \))  which guesses whether or not \( \sigma\concat[m] \in S \).  If it determines \( \sigma\concat[m]  \) is in \( S \) then we go on to \( \module{H}[+]{\sigma\concat[m]} \).  If it determines that \( \sigma\concat[m]  \) is not in \( S \) then no module is assigned above that outcome preventing any path from being constructed.

With this in mind, we can now give a formal definition of \( \hT \) and its stagewise approximation.

\begin{definition}\label{def:hT}
We define \( \hT[s](\sigma) = \Mout{\xi} \)  where \( \xi = (\alpha, \sigma) \in \tpath[s] \) and \( \xi^{-} \) implements the module \( \module{H}{\sigma} \).  If no such \( \xi \in \tpath[s] \) then it is undefined.    We define \( \hT(\sigma)  \)) in a similar manner except we require that \( \xi \in \tpath \). 
\end{definition}

\subsection{Modules}

We now describe the operation of the modules.  For this subsection, we describe the operation of the module assuming it is located at the node \( \xi \in \pTree, \xi = (\alpha, \sigma) \) and executing at stage \( s \).  As we only define the outcome of the module at \( \xi \)  when \( s_\xi > 0 \) we understand the previous outcome of the module to be undefined when \( s_\xi \leq 1 \).

\subsubsection{Module \( \module{P}{e} \)}

The module \( \module{P}{e} \) has outcomes \( \godelnum{\neq} \leftof \godelnum{\diverge} \) and \( (o, \delta) \in \Msucc(\xi) \) iff \( \delta = \estr \) and \( o \) is one of the above two outcomes.  

If the previous outcome was \( \godelnum{\neq} \) we retain that outcome.  Otherwise, the module acts as follows.

Check if there is any \( \tau \in T_s, \tau \supfun \Mout{\xi}  \) with \( \recfnl[s]{e}{X}{} \incompat \tau  \).  If found set the outcome to \( \godelnum{\neq} \) and  \( \Mout{\pair{\alpha\concat[\godelnum{\neq}]}{\sigma}}  \) to be a \( \subfun \) maximal extension of \( \tau \) in \( T_s \).  Otherwise, set the outcome to \( \godelnum{\diverge} \) and, if this is the first stage at which that outcome is visited, set \( \Mout{\pair{\alpha\concat[\godelnum{\diverge}]}{\sigma}} = \Mout{\xi} \).

\subsubsection{Module \( \module{L}{e} \)}

The module \( \module{L}{e} \) has outcomes \( \godelnum{\conv} = 0 \) and \( \godelnum{\nu} \) where \( \nu \supfun \xi \) and some module of the form \( \module{L}[n]{e} \) is assigned to \( \nu \) (hence \( \godelnum{\nu} > 0 \)) .  \( \Msucc(\xi) \) consists of all pairs \( (o, \delta) \) where  \( o \) is one of the allowed outcomes and \( \delta = \estr \).   The outcome \( \godelnum{\conv}  \) corresponds to the state where  \( \recfnl{e}{X \Tplus f}{} \) is total for all paths \( f \in [T], f \supfun \Mout{\xi} \) and the outcome  \( \godelnum{\nu} \) where \( \nu \) implements \( \module{L}[n]{e} \) corresponds to the state where all \( f \in [T], f \supfun \Mout{\pair{\alpha\concat[\godelnum{\nu}]}{\sigma}} \) satisfy \( \recfnl{e}{X \Tplus f}{n}\diverge \).

Intuitively, we can think of the operation of \( \module{L}[n]{e} \) as  creating something of a link with  \( \module{L}{e} \) as in a \( \zerojjj \) construction.  When we visit \( \module{L}[n]{e} \) we effectively pause the operation of all the intervening modules between \( \module{L}{e} \) and \( \module{L}[n]{e} \) and start meeting the modules extending  \( \module{L}{e} \) again until we find find an extension which causes the  \( e \)-th functional to converge on argument \( n \) at which point we return to  \( \module{L}{e} \).  Luckily, however, we don't need the full machinery of links and can achieve this effect merely by letting the  module  \( \module{L}[n]{e} \) manipulate the internal state of the unique visiting module \( \module{L}{e} \) as defined here.

If \( s_\xi = 0 \) we initialize \( \upsilon = \diverge, \delta = \diverge \).  If \( s_\xi > 0 \) and \( \upsilon\diverge \) then we visit the outcome \(  \godelnum{\conv} \) with \( \Mout{\pair{\alpha\concat[\godelnum{\conv}]}{\sigma}} = \Mout{\xi} \).  We leave it to the submodules of the form \( \module{L}[n]{e} \) to define \( \upsilon, \eta \) when necessary. 

If \( s_\xi > 0 \) and \( \upsilon\conv \) with \( \module{L}[n]{e} \) assigned to \( \upsilon \) then we check if there is a (maximal) extension \( \tau \supfun \delta, \tau \in T_s \) such that \( \recfnl{e}{\tau}{n}\conv \).  If there is, then set \( \Mout{\upsilon^{+}} = \tau \) where \( \upsilon^{+} \) is the unique successor of \( \upsilon \) on \( \pTree \), set \( \upsilon, \delta \) to be undefined and visit the outcome  \( \godelnum{\conv} \) as above.  Otherwise, visit the outcome  \( \godelnum{\upsilon} \) with \( \Mout{\pair{\alpha\concat[\godelnum{\upsilon}]}{\sigma}} = \delta \).

\subsubsection{Module \( \module{L}[n]{e} \)}

This node only has a single outcome \( 0 \) and \( \Msucc(\xi) = \set{(0, \estr)}{} \).

Let \( \nu \subfun \xi \) be the unique ancestor node implementing \( \module{L}{e} \).  If \( s_\xi = 0 \) then set the variables \( \upsilon, \delta \) for the module at node \( \nu \) to be equal to \( \xi  \) and \( \Mout{\xi}  \) respectively.    If we are ever visited again, we visit our single outcome and rely on the node implementing \( \module{L}{e} \) to have set \( \Mout{\pair{\alpha\concat[0]}{\sigma}} \).

\subsubsection{Module \( \module{H}[+]{\sigma} \)}

This node only has a single outcome \( 0 \) and \( \Msucc(\xi) = \set{(0, \estr)}{} \).   Let \( \nu = (\alpha\concat[0], \sigma) \) and if \( s_\xi = 1 \) then set  \( \Mout{\nu} \) to be a maximal element in \( T_s \) extending \( \Mout{\xi} \).  If \( s_\xi = n + 1 \) (hence \( s_\nu = n \) ) then let \( k \) be large and \( \tau = \Mout{\nu}\concat[k]  \)  (in particular, large enough that if \( \godelnum{\tau} > s \)).   Enumerate \( \tau  \) into \( T_{s+1} \) and set \( \Mext[n]{\nu} = \tau \).   

\subsubsection{Module \( \module{S}[n]{-1} \)}  

This node only has a single outcome \( 0 \) but \( \Msucc(\xi) = \set{(0, \str{n}), (0, \estr)}{} \).  This module doesn't take any actions, merely split up input it gets between the two successor nodes as follows.  Specifically, it sets \( \Mout{\pair{\alpha\concat[0]}{\sigma\concat[n]}} = \Mext[0]{\xi} \),  \( \Mout{\pair{\alpha\concat[0]}{\sigma}} = \Mout{\xi} \) and \( \Mext[n]{\pair{\alpha\concat[0]}{\sigma}} = \Mext[n+1]{\xi} \).

\subsubsection{Module \( \module{S}[n]{e} \)}\label{ssec:module-S-n-e}

This node has outcomes \( \godelnum{\compat} = 0 \), \( \godelnum{\incompat_0} = 1 \), \( \godelnum{(\incompat_1, n)} = 2 + \pair{n}{0} \), \( \godelnum{(\diverge,n,m)} = \pair{n}{m+1}  \).  \( \Msucc(\xi) \) consists of all pairs \( (o, \delta) \) where  \( o \) is one of the allowed outcomes and \( \delta = \estr \).   Remember that in what follows \( e \)-splitting refers to the notion relativized to \( X \).   

The outcome \(  \godelnum{\compat} \) corresponds to the case where \( \Mext[0]{\xi} \) isn't extended by an \( e \)-splitting in \( T \) (as \( \compat \) indicates the strings are compatible).  The other outcomes presume we do find some \( e \)-splitting \( \tau_0, \tau_1 \) extending \( \Mext[0]{\xi} \).     The outcome \( o = \godelnum{\incompat_0} \) corresponds to the case where we find infinitely many elements in \( \Mexts{\xi} \) that extend to an \( e \)-splitting with \( \tau_0 \).  Outcomes of the form \( \godelnum{(\incompat_1, n)}  \) correspond to the case where we only find \( n \)  elements in \( \Mexts{\xi} \) that extend to an \( e \)-splitting with \( \tau_0 \) but infinitely many which extend to an \( e \)-splitting with \( \tau_1 \). Finally, the outcomes of the form \( \godelnum{(\diverge,n,m)}  \) correspond to the case where we find \( n \) elements in \( \Mexts{\xi} \)   extending to \( e \)-splittings with \( \tau_0 \) after which we find another \( m \) elements extending to an \( e \)-splitting with \( \tau_1 \) but infinitely many elements don't extend to an \( e \)-splitting with either.   

Let  \( \nu =  \pair{\alpha\concat[o]}{\sigma} \) for whatever value we specify for the outcome \( o \).  We'll ensure that \( \Mext[0]{\nu} \) \( e \)-splits with \( \Mext[n+1]{\nu}, n \in \omega \) in all cases except when \( o =  \godelnum{\compat} \) or \( o =  \godelnum{(\diverge,n,m)} \).  In the later case, we'll ensure that neither \( \tau_0 \) or \( \tau_1 \)  \( e \)-split with any extension of \( \Mext[n+1]{\nu}, n \in \omega \)  (ensuring that if \( f \supfun \Mext[n+1]{\nu} \) then \( \recfnl{e}{f \Tplus X}{}\diverge \)).  Since we don't want to accidentality extend \( \Mext[0]{\pair{\alpha\concat[o']}{\sigma}} \) to an infinite path if \( o' \) isn't the true outcome we'll ensure that every outcome except \( \godelnum{\compat}  \) corresponds to an incompatible value for \( \Mext[0]{\pair{\alpha\concat[o']}{\sigma}} \) of length at most \( 1 + \max \lh{\tau_0}, \lh{\tau_1} \) while always ensuring that \( \Mext[0]{\nu} \supfun \Mext[0]{\xi} \).

We define \( \Mout{\nu} = \Mout{\xi} \) for all potential outcomes \( o \).  When \( s_\xi = 0 \) we start by setting \( \tau_0, \tau_1 \) to be undefined and \( \hat{n} = \hat{m} = 0  \).  For \( s_\xi > 0 \) we consider the following cases. 

\begin{pfcases*}
    \case[\( \tau_0\diverge \)] Check if there are  \( \tau_0, \tau_1 \in T_s \) with \( \tau_0, \tau_1 \) \( e \)-splitting extensions of \( \Mext[0]{\xi} \).  If no such values are found, then visit outcome  \( o = \godelnum{\compat} \) and define \( \Mexts{\nu} = \Mexts{\xi} \).  If such values are found, let \( \tau_0, \tau_1 \) be \( \subfun \)  maximal extensions in \( T_s \) of these \( e \)-splitting extensions of \( \Mext[0]{\xi} \), let \( \hat{m} = \hat{n} = n \) and visit outcome  \( o = \godelnum{\incompat_0}  \) setting \( \Mext[0]{\nu} = \tau_0 \).   

    \case[\( \tau_0\conv  \) ] We break this up into a number of subcases.  We search for some \( n \leq s_\xi \) and \( \tau' \supfun \Mext[n]{\xi}, \tau' \in T_s  \) that satisfy the following (picking the first case satisfied) 

        \begin{pfcases*}
            \case[\( \tau', \tau_0 \) \( e \)-split with  \( \hat{n} < n \leq s_\xi \)]  In this case, we let \( o = \godelnum{\incompat_0}  \) and set \( \Mext[s_\nu]{\nu} \) to be a \( \subfun \) maximal extension of \( \tau' \) in \( T_s \).  Finally, we set \( \hat{n} = \hat{m} = n \).

            \case[\( \tau', \tau_1 \) \( e \)-split with  \( \hat{m} < n \leq s_\xi \)]  In this case, we let \( o = \godelnum{(\incompat_1, \hat{n})} \) and set \( \hat{m} = n \) and \( \Mext[s_\nu +1]{\nu}  \) be a \( \subfun \) maximal extension of \( \tau' \) in \( T_s \).   

            If this is the first time we've visited this outcome \( s_\nu = 0 \).  We pick \( k \) to be larger than any number mentioned so far in this construction, and set \( \Mext[0]{\nu} =\tau_1\concat[k] \) placing \( \Mext[0]{\nu} \) into \( T_{s+1} \).   

            \case[Otherwise] In this case, we visit outcome \( o = \godelnum{(\diverge,\hat{n},\hat{m})} \).  If this is the first time we've visited this outcome, set \( \Mext[0]{\nu} =\tau_1\concat[k] \).  Let \( \tau' \) be a \( \subfun \) maximal element of \( T_s \) extending \( \Mext[\hat{m} + s_\nu + 1]{\xi} \) and set \( \Mext[s_\nu+1]{\nu} = \tau' \).          

        \end{pfcases*}     

\end{pfcases*}

\subsubsection{Module \( \module{H}{\sigma} \)}

We note that we can assume (see Lemma \ref{lem:limit-lemma}) we have a uniformly given total computable binary valued function (indeed functional) \( \rho(\sigma, s_1, s_0) \) such that \( \rho_1(\sigma, s_1) = \lim_{s_0 \to \infty}  \rho(\sigma, s_1, s_0)  \) is total, \( S(\sigma) = \lim_{s_1 \to \infty} \rho_1(\sigma, s_1) \) (both diverging if either does).   

Morally speaking, this module has the outcomes \( \godelnum{(i, \hat{n})} \) ordered lexicographically  where \( i \in \set{0,1}{} \) indicates whether the module guesses that \( \sigma \nin S \) or \( \sigma \in S \) and \( \hat{n} \) indicates the value at which \( \rho_1 \) achieves it's limit.   However, to ensure that we never reinitialize any node as required by Condition \ref{cond:no-reinit} we also record a value \( m \) giving the number of times an outcome to the left of \( (i, \hat{n}) \) has been visited.  Thus, the actual outcomes will be of the form  \( \godelnum{(i, \hat{n}, m)} = 2\pair{\hat{n}}{m} +i \).   As our pairing function is strictly monotonic in both arguments this functions just as if we'd used outcomes of the form \( \godelnum{(i, \hat{n})} \) and reinitialized outcomes after passing to their left.    \( \Msucc(\xi) \) consists of all pairs \( (o, \delta) \) where  \( o \) is one of the allowed outcomes and \( \delta = \estr \).

If \( s_\xi = 0 \) set \( \tau \) to be a \( \subfun \) maximal extension of \( \Mout{\xi} \) in \( T_s \).  If \( s_\xi > 0 \) we define \( k  \) to be the number of times before stage \( s \) at which an outcome of the form \( \godelnum{(i, \hat{n}, m)} \) has been visited.  Choose the lexicographically least pair \( (i, \hat{n}) \) such that for all \( n \in [\hat{n}, \hat{n} + k] \) there are \( k + 1 \) distinct values \( x^{n}_j \leq s, j < k +1 \) such that \( i = \rho(\sigma, n, x^{n}_j) \).  Note that, such a pair is always found since for large enough \( \hat{n} \) we have \( k = 0 \) and \( \rho(\sigma, \hat{n}, 0) \in \set{0,1}{} \).  We now visit the outcome \( o = \godelnum{(i, \hat{n}, m)} \) where \( m \)  is the number of times that we've visited an outcome of the form \( \godelnum{(i', n, m)} \) with \( (i',n) \) lexicographically before \( (i, \hat{n}) \) before stage \( s \).  

If this is the first time we've visited outcome \( o \) then pick \( c \) to be large,enumerate \( \tau\concat[c] \) into \( T_{s+1} \)  and set \( \Mout{\pair{\alpha\concat[o]}{\sigma}} = \tau\concat[c] \).

\subsection{Verification}

Before we verify the individual requirements, we verify that the construction controls the paths through \( T \) in the manner desired.

\begin{lemma}\label{lem:path-single-modules}
Suppose that \( \xi \in \tpath \) implements a module of the form \( \module{P}{e}, \module{L}{e}, \module{L}[n]{e}, \module{H}{\sigma} \) and that for some \( \nu \) with \( \nu^{-} = \xi \) we have \( \lh{\tau} = \lh{\Mout{\nu}} \) but \( \tau \nsupfun \Mout{\hat{\xi}} \) for \( \hat{\xi} \in \tpath \), \( \hat{\xi}^{-} = \xi \).  Then there are only finitely many stages at which any module at \( \hat{\nu} \supfun \xi \) enumerates an element \( \tau' \supfun \tau \) into \( T \).                  
\end{lemma}
Note that this covers the case where \( \module{H}{\sigma} \) doesn't have any extension \( \nu \) in the truepath because \( S \) doesn't converge on \( \sigma \).   
\begin{proof}
By Condition \ref{cond:mod-seperation} and Condition \ref{cond:output} (and the fact that no single module ever adds a full path) it is enough to show that there are only finitely many stages at which we visit a node \( \hat{\nu} \) with \( \hat{\nu} \supfunneq \xi \) and  \( \Mout{\hat{\nu}} \) compatible with \( \tau \).     

For the module \( \module{L}[n]{e} \) this is trivial as this module only has a single outcome.   For the module \( \module{H}{\sigma}  \) we note that each time visit to an outcome \( \godelnum{i, \hat{n}, m} \) ensures that all outcomes to the right visit strings that have never been visited before.  As \( \module{P}{e} \) can act at most once this case is also straightforward.  

This leaves only the case \( \module{L}{e} \).  If this module takes any of the finite outcomes the claim is evident and if this module takes the outcome \( \godelnum{\conv} \) then the claim follows because \( \Mout{\nu} \supfun \Mout{\hat{\xi}} \) for all \( \nu \) with \( \nu^{-} = \xi \) when \( \hat{\xi} \) corresponds to the infinitary outcome.      
\end{proof}

\begin{lemma}\label{lem:path-mext-modules}
Suppose that \( \xi \in \tpath \) and \( \xi^{-} \)  implements a module of the form  \( \module{H}[+]{\sigma} \) or \( \module{S}[n]{e}, e \geq 0 \) then for each \( k \) there is some \( l \)  such that if  \( \tau \supfun \Mout{\xi}\concat[k], \lh{\tau} \geq l \) but \( \tau \nsupfun \Mext[n]{\xi} \) for any \( n \) then there are only finitely many stages at which some \( \nu \supfun \xi^{-} \) enumerates an extension of \( \tau \).
\end{lemma}
\begin{proof}
This is trivial if \( \xi^{-} \) implements a module of the form \( \module{H}[+]{\sigma} \).  Also, if there is no \( \Mext[n]{\xi} \supfun \Mout{\xi}\concat[k] \) then there is some last stage at which we visit any \( \nu \) with \( \nu^{-} = \xi^{-} \) with \( \Mexts{\nu} \) containing an extension of \( \Mout{\xi}\concat[k]  \).  So suppose that \( \Mext[n]{\xi} \supfun \Mout{\xi}\concat[k]  \).  If \( n > 0 \) then, once we set \( \Mext[n]{\xi} \) we never visit any \( \nu \neq \xi \) with \( \nu^{-} = \xi^{-} \) with \( \Mext[m]{\xi} \supfun \Mout{\xi}\concat[k]  \).  

This leaves only the case where \( \Mext[0]{\xi} \supfun \Mout{\xi}\concat[k] \).  If \( \xi \) corresponds to an outcome \( \godelnum{(\diverge,n,m)} \) or \( \godelnum{\compat} \)  then we never again visit another extension of \( \xi^{-} \) after visiting \( \xi \) so the bound can be deduced by looking at the finitely many stages before that happens.  Thus, we can assume \( \xi \) corresponds either to   \( \godelnum{\incompat_0}  \) or  \( \godelnum{(\incompat_1, n)} \).  In this case, we can take \( l = 1 + \max \lh{\tau_0}, \lh{\tau_1} \) by the remark in Section \ref{ssec:module-S-n-e}.

\end{proof}



\begin{lemma}\label{lem:ht-yields-paths}
For all \( f \in \baire, \) \( f \in [T] \) iff \( \exists(g \in \baire)\left(\hT(g) = f \right) \). 
\end{lemma} 
\begin{proof}
For the if direction suppose that \( \hT(g) = f \).  Since any module \( \module{M}{\xi^{-}} \) which sets \( \Mout{\xi} = \sigma \) guarantees that \( \sigma \in T \) (and we always have \( \sigma^{-} \in T \)) and we have that \( \hT(g) \in T \).  For the other direction, suppose \( f \neq \hT(g) \) for any \( g \).  Let  \( \xi \) be a \( \subfun \) minimal element such that \( \Mout{\xi^{-}} \subfun f \) and, if \( \Mexts{\xi^{-}} \) defined, with \( f \) extending some element in \( \Mexts{\xi^{-}} \) but one of those conditions failing for \( \xi \).  We then argue using  Lemmas \ref{lem:path-single-modules} and \ref{lem:path-mext-modules} that there is some \( \tau \subfun f \) such that only finitely many extensions of  \( \tau \) are enumerated into \( T \).

The only other real difficulty occurs with nodes of the form \( \module{L}{e} \) since those are the only modules that might be assigned to some \( \nu \subfunneq \xi^{-} \)  where the above lemmas (or, for \( \module{L}{e} \), fact that we only visit incorrect outcomes finitely many times) don't directly exclude the possibility that some module extending an outcome to the right of the true outcome of \( \nu \) could contribute infinitely many elements compatible with \( f \).  Unless the outcome of \( \nu \) is \( \godelnum{\conv} \) this clearly can't happen as in those cases we settle on the true outcome after finitely many stages.  However, there will only be finitely many stages at which \( \nu \) visits an outcome of the form \( \godelnum{\eta} \) where \( \Mout{\eta} \subfun f \) ensuring that the false outcomes of \( \nu \) don't cause a problem. 

Note that, all modules excepting those of the form  \( \module{H}{\sigma} \)  clearly have a well-defined leftmost outcome that's visited infinitely often and by Lemma \ref{lem:path-single-modules}, in the case where \( \xi \) implements \( \module{H}{\sigma} \) and  \( \xi \) has  no successor along the truepath  is also a case where \( \Mout{\xi} \) can't be extended to any path through \( T \). 
\end{proof}

With this result in hand, we can now verify the properties claimed in Proposition \ref{prop:tree-build}.

\begin{lemma}\label{lem:tree-build:rng}
\( \rng \hat{T} \subset T \) and \( [\hat{T}] = [T] \), that is, claim \ref{prop:tree-build:hat-t} of Proposition \ref{prop:tree-build} holds.
\end{lemma} 
\begin{proof}
By Condition \ref{cond:output} anytime a module sets \( \Mout{\xi} = \tau \) it ensures that \( \tau \in T \).  This ensures the first part of the claim holds.  The second half of the claim is just Lemma \ref{lem:ht-yields-paths}.   
\end{proof}

\begin{lemma}\label{lem:module-H-is-correct}
If \( \xi \in \tpath \) and \( \xi \)  implements  \( \module{H}{\sigma} \) then \( \xi \) has a well-defined outcome iff \( S(\sigma)\conv \) and that outcome is always correct about the membership of \( \sigma  \) in \( S \).
\end{lemma}
By well-defined outcome we mean a leftmost outcome that is visited infinitely often.

\begin{proof}
Suppose that the module at \( \xi \) has the true outcome \( \godelnum{(i, \hat{n}, m)} \).  If \( S(\sigma) \neq i \) (including divergence) then, since \( \rho_1(\sigma, s_1) \) (first limit), can be taken to be always defined, then for some \( n > \hat{n} \) we have \( \rho_1(\sigma, n) =  1 -i  \).  Thus, for some  \( k \) we have \( \rho(\sigma, n, k') =  1 - i \) for all \( k' \geq k \) contradicting the assumption that we visit this outcome more than \( n + k + 1 \) times.  Thus, \( \module{H}{\sigma} \) is never incorrect and thus must not have an outcome whenever \( S(\sigma)\diverge \).

Now suppose that \( S(\sigma)\conv = i \).    For some minimal \( \hat{n} \) we have \( \rho_1(\sigma, s_1) = i \) for all \( s_1 \geq \hat{n} \).  By minimality, there is some last stage at which any outcome of the form \(\godelnum{(i, n, m)}   \) with \( n < \hat{n} \) is visited and as \( \rho_1(\sigma, \hat{n}) = i \) there is some last stage at which any outcome of the form \(\godelnum{(1 - i, \hat{n}, m)}   \)  is visited.  Thus, after some point we never visit an outcome of the form \(\godelnum{(i', n, m)}   \) with \( (i', n) \) lexicographically before \( (i, \hat{n}) \) and thus there is some \( m \) for which  \(\godelnum{(i, \hat{n}, m)}   \) is the true outcome. 
\end{proof}

\begin{lemma}\label{lem:tree-build:T-is-image}
For all \( g \in \baire \), \( \hT(g) \) is total iff \( g \in [S] \) iff \( \hT(g) \in [T] \).  Moreover, \( \hT(\sigma)\conv \) iff \( S(\sigma)\conv \) and all \( \sigma' \subfunneq \sigma \) are in \( S \).   

\end{lemma}
\begin{proof}
We first verify the moreover claim.  Note that, for any module besides \( \module{H}{\sigma} \) there is always is a well-defined true outcome.  Thus, by an examination of the construction, we can see that the only way for  \( \hT(\sigma) \) to be undefined is either if for some \( \sigma' \subfunneq \sigma \) the module of the form  \( \module{H}{\sigma'} \) on the truepath doesn't have a true outcome guessing \( \sigma' \in S \) or if the module    \( \module{H}{\sigma} \) on the truepath doesn't have any true outcome.  Clearly, if either of those cases obtain then we actually do have \( \hT(\sigma)\diverge \) so this result follows from Lemma \ref{lem:module-H-is-correct}    

The main claim follows trivially since  \( g \in [S] \) iff all \( \sigma \subfun g \) are elements in \( S \).       
\end{proof}

\begin{lemma}\label{lem:tree-build:w-branching}
\( \hT \) is an f-tree  
\end{lemma}
\begin{proof}
As \( \hT \) is clearly \( \subfun \) respecting it is enough to show that whenever \( \hT(\sigma) \) isn't terminal then \( i < j \) implies that  \( \hT(\sigma\concat[i]) \) and \( \hT(\sigma\concat[j]) \) extend incompatible immediate extensions of  \( \hT(\sigma) \)  and \( \hT(\sigma\concat[i]) \) is lexicographically below \( \hT(\sigma\concat[j]) \).  However, this is immediate from the operation of  \( \module{H}[+]{\sigma} \) and the fact that nodes of the form \( \module{S}[n]{e} \) maintain these properties.

\end{proof}

We can use this to prove the homeomorphism claim from Proposition \ref{prop:tree-build}.

\begin{lemma}\label{lem:tree-build:homeo}
  \( \hat{T} \) is a homeomorphism of \( [S] \) with \( [T] \).  That is claim \ref{prop:tree-build:homeo} of Proposition \ref{prop:tree-build} holds.  
\end{lemma}
\begin{proof}
By Lemma \ref{lem:tree-build:T-is-image} we know that \( [T] \) is the image of \( [S] \) under \( \hT \).  Evidently, both \( \hT \) and its inverse are continuous so it remains only to show that \( \hT \) is injective.  However, this follows from Lemma \ref{lem:tree-build:w-branching}.

\end{proof}

\begin{lemma}\label{lem:tree-build:comp}
If \( \Upsilon_2 \) is a computable functional then we can uniformly find computable functionals \( \Upsilon, \hat{\Upsilon} \) such that if \( \Upsilon_2(\jjump{X}) = S \) then \( \Upsilon(X) = T \) and \( \hat{\Upsilon}(\jjump{X}) = \hT \) where \( T, \hT \) are as constructed above.     
\end{lemma}
\begin{proof}
To compute \( \hT(\sigma) \) from \( \jjump{X} \) we simply (iteratively) identify the leftmost outcome of nodes on \( \pTree \) to identify elements in \( \tpath \)  and search for a node  \( \xi \in \tpath \) and \( \xi^{-} \) implementing \( \module{H}{\sigma} \)  and return \( \Mout{\xi} \).  It's possible that when working to compute \( \hT(\sigma) \) we next discover such a node \( \xi \).  However, this can only happen when \( S \) fails to converge on some  \( \sigma' \subfun \sigma  \) in which case \( \hT(\sigma) \) is properly undefined anyway.   The uniformity can be read off the construction (note the only use of \( S \) is via Lemma \ref{lem:limit-lemma} which is fully uniform).    
\end{proof}

\begin{lemma}\label{lem:tree-build:non-compute}
   If \( f \in [T] \) then \( f \nTleq X \).  That is claim \ref{prop:tree-build:non-compute} of Proposition \ref{prop:tree-build} holds.  
\end{lemma}
\begin{proof}
    Suppose the claim fails and for some \( f \in [T], f = \recfnl{e}{X}{} \) for some total \( \recfnl{e}{X}{} \) .  By Lemma \ref{lem:ht-yields-paths} we have \( f = \hT(g) \).  Thus, for some \( \xi \in \tpath \)  with \( \xi^{-} \)  implementing \( \module{P}{e} \) we have  \( f \supfun \Mout{\xi} \).   As \( f = \hT(g) \) we know that \( \tpath \) extends to a node implementing some \( \module{H}[+]{\sigma} \) and thus \( T \) contains incompatible \( \tau_0, \tau_1 \supfun \Mout{\xi^{-}} \).  Thus, by the operation of \( \module{P}{e} \) we know that \( \Mout{\xi} \incompat \recfnl{e}{X}{}  \) contradicting the supposition.   
\end{proof}

\begin{lemma}\label{lem:tree-build:low2-reqs-work}
Suppose \( \xi \in \tpath, f \in [T], f \supfun \Mout{\xi} \) and \( \xi^{-} \) implements \( \module{L}{e} \) then \( \recfnl{e}{f \Tplus X}{} \) is total iff the true outcome of \( \xi^{-} \) is \( \godelnum{\conv} \)    
\end{lemma}
\begin{proof}
Suppose \( \xi^{-} \)  has true outcome \( \godelnum{\conv} \) but that \( \recfnl{e}{f \Tplus X}{n}\diverge \).  For some  \( \nu \in \tpath, \tau = \Mout{\nu}  \) with \( \nu^{-} \supfun \xi \) implementing \( \module{L}[n]{e} \) we have \( f \supfun \tau \) and by operation of  \( \module{L}{e} \) and \( \module{L}[n]{e} \) we can only have true outcome  \( \godelnum{\conv} \)  if \( \recfnl{e}{\tau \Tplus X}{n}\conv \).  This contradicts our assumption.

On the other hand, if  \( \xi^{-} \)  has true outcome \( \godelnum{\nu} \) then \( \nu \supfun \xi^{-} \) implements some \( \module{L}[n]{e} \).  The operation of \( \module{L}{e} \) guarantees that if we ever saw some \( \tau \supfun \Mout{\xi}, \tau \in T \) with  \( \recfnl{e}{\tau \Tplus X}{n}\conv \) then we wouldn't have true outcome \( \godelnum{\nu} \).  Hence, we must have  \( \recfnl{e}{f \Tplus X}{n}\diverge \).  As the operation of \( \module{L}{e} \) ensures that one of our outcomes is true, this establishes the claim.

\end{proof}

\begin{lemma}\label{lem:tree-build:low2}
   If \( g \in [S] \) then \( g \Tplus \jjump{X} \Tequiv \jjump{\left(\hT(g) \Tplus X\right)} \Tequiv \hT(g) \Tplus   \jjump{X}  \).  That is claim \ref{prop:tree-build:low2} of Proposition \ref{prop:tree-build} holds.  
\end{lemma}
Note that, as \( [T] \) is the image of \( [S] \) under \( \hT \) every \( f \in [T] \) has this property relative to some \( g \).
\begin{proof}
Let \( g \in [S] \).  By Lemma \ref{lem:tree-build:homeo}, \( \hT(g) = f \) for some total \( f \).  Since whenever \( \hT(\sigma)\conv \) by Lemma \ref{lem:tree-build:comp} we can find \( \hT(\sigma) \) computably in \( \jjump{X} \) it follows that \( f \Tleq g \Tplus \jjump{X} \).  To see that \( g \Tleq f \Tplus \jjump{X} \) note that, by Lemmas \ref{lem:tree-build:comp} and \ref{lem:tree-build:w-branching} we can inductively recover the unique path \( g \) with  \( \hT(g) = f \) computably in \( f \Tplus \jjump{X} \).   

Clearly  \( \jjump{\left(f \Tplus X\right)} \Tgeq f \Tplus \jjump{X} \).  Thus, to complete the proof, it is sufficient to show that \( \jjump{\left(f \Tplus X\right)} \Tleq g \Tplus \jjump{X}  \).  By a well-known result, it is enough to show that \( g \Tplus \jjump{X} \) can computably decide whether \( e \) is an index for a total computable function in \( f \Tplus X \).  However, by Lemma \ref{lem:tree-build:low2-reqs-work} we can decide this question by searching for \( \xi \in \tpath \) with \( \xi \) implementing \( \module{L}{e} \) and \( \xi = \pair{\alpha}{\sigma} \) with \( \sigma \subfun g \) and determining the outcome of \( \xi \).  By Lemma \ref{lem:tree-build:comp} this can be done computably in  \( g \Tplus \jjump{X} \).

\end{proof}   

\begin{lemma}\label{lem:no-esplit-ext-computable} 
If  \( f \in [T], Y = \recfnl{e}{f \Tplus X}{} \) and there is some \( \tau \subfun f \) with no \( e \)-splitting \( \tau_0, \tau_1 \supfun \tau \) with \( \tau_0, \tau_1 \in T \)  then \( Y \Tleq X \).     
\end{lemma}
As remarked above, we mean \( e \)-splitting relativized to \( X \).  
\begin{proof}
We can compute \( Y\restr{n} \) from \( X \) by returning  \( \recfnl{e}{\tau' \Tplus X}{}\restr{n} \) for the first \( \tau' \supfun \tau \) in \( T \) we can find for which this is a string of length \( n \).  As \( Y = \recfnl{e}{f \Tplus X}{}  \) and \( f \supfun \tau , f \in [T] \) there must be some such \( \tau' \) and by the lack of \( e \)-splitting extensions there is no possibility of an incompatible value. 
\end{proof}

\begin{lemma}\label{lem:tree-build:min}
   If \( f \in [T] \) and \( Y \Tleq f \Tplus X  \) then either \( Y \Tleq X \) or \( f \Tleq Y \Tplus \jjump{X} \).  That is claim \ref{prop:tree-build:min} of Proposition \ref{prop:tree-build} holds.  
\end{lemma}
\begin{proof}
Suppose that \( Y = \recfnl{e}{f \Tplus X}{} \) and, by Lemma \ref{lem:tree-build:homeo},  that \( f = \hT(g) \).    Let \( r(\sigma, n) \) where \( \lh{\sigma} > e \)  be the outcome of the module \( \module{S}[n]{e} \) assigned to some \( (\alpha, \sigma) \in \tpath \) and \( m(\sigma, n) = \xi \in \tpath \) where  \( \xi^{-} = (\alpha, \sigma) \).   Note that, for all \( n, \sigma \in \dom \hT \) with \( \lh{\sigma} > e \),  \( \hT(\sigma\concat[n]) \supfun \Mext[0]{m(\sigma, n)}  \).  This is because all modules of the form  \( \module{S}[n]{e}, e \geq 0 \) always output a leftmost branch extending the leftmost branch they receive as input.

If there is some \( \tau \subfun f \) not extended by any \( e \)-splitting extensions then by Lemma \ref{lem:no-esplit-ext-computable} we are done.  So we may suppose this isn't the case and show that this implies \( g \Tleq Y \Tplus \jjump{X} \) which, by Lemma \ref{lem:tree-build:low2} is equivalent to showing \( f \Tleq Y \Tplus \jjump{X}  \).

Suppose we know \( \sigma, n' \) and that \( \sigma\concat[n] \subfun g \) for some \( n \geq n' \)  (where \( \lh{\sigma} > e \)) we demonstrate how to check if \( n = n' \) or \( n > n' \) computably in  \( Y \Tplus \jjump{X} \).    This will suffice since, applying this repeatedly, we can compute \( g \) from   \( Y \Tplus \jjump{X} \).  

Using \( \jjump{X} \) we can compute \( \xi = m(\sigma, n') \) and \( r(\sigma, n') \).  If \( n = n' \) then we would have \( f \supfun \Mext[0]{\xi} \) while if \( n > n' \) then \( f \supfun \Mext[k]{\xi}, k > 0 \).

If we have \( r(\sigma, n') = \godelnum{\compat} \) then that entails \( \Mext[0]{\xi^{-}}  \)  has no \( e \)-splitting extensions.  Thus, by our supposition \( f \) can't extend \( \Mext[0]{\xi^{-}}  \), so we can conclude \( n \neq n' \).  

Suppose instead that \( r(\sigma, n') = \godelnum{(\diverge, \hat{n},\hat{m})} \).  In this case, the module  \( \module{S}[n']{e} \) at \( \xi^{-} \) must have identified some \( e \)-splitting   \( \tau_0, \tau_1 \) of \( \Mext[0]{\xi^{-}} \) and that no \( \Mext[k]{\xi}, k > 0 \) has an extension in \( T \) which \( e \)-splits with either \( \tau_0 \) or \( \tau_1 \).  However, as \( \recfnl{e}{f \Tplus X}{} \) is total, if \( f \) doesn't extend either \( \tau_0 \) or \( \tau_1 \) then some initial segment of \( f \) must \( e \)-split with either \( \tau_0 \) or \( \tau_1 \) (\( \recfnl{e}{f \Tplus X}{} \) can't agree with incompatible strings).  Hence, we can't have \( f \supfun \Mext[k]{\xi}, k > 0 \) so we must have \( n = n' \).        

This leaves only the case in which \( r(\sigma, n') \) gives one of the incompatible (i.e. \( e \)-splitting) outcomes.  In this case, we simply test if \( Y \) is compatible with \( \recfnl{e}{\Mext[0]{\xi} \Tplus X}{} \).  If so, then \( n = n' \).  If not, then \( n > n' \).

\end{proof}

These lemmas, taken together, verify all parts of Proposition \ref{prop:tree-build} except \ref{prop:tree-build:len} which is evident from the construction.

\appendix
\section{Minimality Requires Double Jump}\label{app:min-double-jump}

Here we present the promised result about the need for double, rather than single, jump inversion from Section \ref{sec:min-and-double-jump-inversion}

\begin{prop:no-do-one-jump}
Given a perfect weakly \( \omega \)-branching pruned f-tree \( T \Tleq \zeroj  \) one can uniformly construct a computable functional  \( \recfnl{}{}{} \) such that \( e \)-splitting pairs in \( T \)  occur above every node in \( T \) and for every \( \tau \in \rng T \)  there   are paths \( f \neq g \) extending \( \tau \)  through \( T \) with \( \recfnl{e}{f}{} = \recfnl{e}{g}{} \).    
\end{prop:no-do-one-jump}

Indeed, as remarked in that section, we actually prove a slightly stronger result and show that we even if \( T \) is unpruned we can start building such paths \( f, g \) above any node \( \tau \in T \) and maintain agreement under \( \recfnl{e}{}{} \) with the only potential for failure being the possibility of hitting a terminal node in \( T \).     

\begin{proof}
The basic idea of this proof is to use the fact that \( T(\sigma) \) has infinitely many immediate extensions on \( T \) to define a limiting behaviour for \( \recfnl{e}{\tau}{} \) for \( \tau \supfun T(\sigma\concat[m]) \) for sufficiently large \( m \).   

Let \( T_{s} \) be a stagewise approximation to \( T \) that's correct in the limit and which doesn't converge on elements outside the domain.  WLOG we may assume that if \( T_s(\sigma) = \tau \) then \( \godelnum{\sigma}, \godelnum{\tau} < s \) and \( T_s(\sigma')\conv \) for all \( \sigma' \subfun \sigma \).  We  say that  \( T(\sigma) \) was defined at \( s' \) (relative to \( s \)) if \( s' = \mu{t}{\forall(t' \in [t, s]) \left(T_{t'}(\sigma)=T_s(\sigma)\right)} \) and that \( \sigma \) is senior to \( \sigma' \)  at stage \( s \) if  \( T(\sigma) \) was defined at an earlier stage than \( T(\sigma') \) or the same stage and \( \sigma \leftof \sigma' \).  We'll also talk about about \( T(\sigma) \) being senior to \( T(\sigma') \) when \( \sigma \) is senior to \( \sigma' \)  .   

We first ensure that every non-terminal \( T(\sigma) = \tau \) is extended by an e-splitting.  The idea here is that 
if \( \tau' \supfun \tau \) we define \( \recfnl{e}{\tau'}{} \) to extend \( \recfnl{e}{\tau}{} \) with a guess at the first prior stage that some extension of \( \tau \) is permanently seen to be \( T( \sigma\concat[m]) \)  for some \( m \). Eventually, this guess stabilizes and will differ from the guess that was made along the most senior extension of \( \tau \).   

Given \( T_s(\sigma) = \tau \),  define \( q_s(\tau) = t  \) where \( \sigma\concat[m] \) is the most senior immediate extension of \( \sigma \) relative to \( s \) and \( t \) the stage it was defined  or \( 0 \) if no such \( t \) exists.  We let \( q(\tau) = \lim_{s \to \infty} q_s(\tau) \).    If \( s = \godelnum{\tau} \) we define \( \recfnl{e}{\tau}{} = \recfnl{e}{\tau^{-}}{}\concat[q_s(\tau^{-})]\concat l_s(\tau^{-}) \) where \( l_s(\tau) \) is a computable function that will be defined for the second part of the proof. 

Note that, regardless of the behaviour of \( l_s(\tau) \) this guarantees that every non-terminal \( T(\sigma) \) is extended by an e-splitting.  Given \( \sigma \) with \( T(\sigma) = \tau^{-} \) and \( \sigma\concat[m], t \) witnessing that \( q(\tau^{-}) = t  \) we must have \( t > \godelnum{\sigma\concat[m]} \).  Thus, if  \( \tau \subfun  T(\sigma\concat[m] ) \) we have that \( \recfnl{e}{\tau}{} \supfun \recfnl{e}{\tau^{-}}{}\concat[q_s(\tau^{-})] \) for \( s = \godelnum{\sigma\concat[m]} \) and thus \( q_s(\tau^{-}) <  q(\tau^{-}) \).  As \( T \) is weakly \( \omega \)-branching for some sufficiently large \( m' > m  \) we have \( T(\sigma\concat[m']) = \tau' \)  and \( \recfnl{e}{\tau'}{} \supfun \recfnl{e}{\tau^{-}}{}\concat[q(\tau^{-})] \).    

We now abstract away from the construction we just performed by noting that we can effectively define a subtree of \( T \) consisting of the restriction of \( T \) to those nodes of the form  \( \sigma\concat[m] \) with \( \godelnum{\sigma\concat[m']} \geq q(T(\sigma)) \).  On this subtree \( q_s(\tau^{-}) \) has achieved it's limit anytime it is used to define \( \recfnl{e}{\tau}{}  \).  For the remainder of this proof we therefore work on this subtree and (by redefinition) assume that \( T_s(\sigma) = \tau \) doesn't permanently settle on a value until \( q_s(\tau^{-}) \) reaches it's limit.  This allows us to use  \( \recfnl{e'}{\tau^{-}}{}  \) to refer to the value \(  \recfnl{e}{\tau^{-}}{}\concat[q(\tau)] \) and assume that it's correct at any stage at which \( \tau^{-} \) has permanently entered \( \rng T \).  We can also assume that if \( T_s(\sigma)\conv \neq T_{s+1}(\sigma) \) then  \( T_{s+1}(\sigma)\diverge  \) and  \( T_{s}(\sigma) \) isn't extended by any element in \(  \rng T_{s+1} \).  

The idea now is still to pick a limiting behaviour of \( \recfnl{e'}{\tau'}{} \) on extensions \( \tau' \supfun T(\sigma\concat[m]) \) for large enough values of \( m \).  If we've defined some initial segment \( f^{n+1}_\sigma \) of some path through \( T \) whose image under \( \recfnl{e'}{}{} \) extends that of \( g^{n}_\sigma \) we can control the limiting behaviour of extensions of \( g^{n}_\sigma \) to ensure that we can eventually find some extension \( g^{n+1}_\sigma \) whose image under \( \recfnl{e'}{}{} \) extends that of \( f^{n+1}_\sigma \).  By iterating this, we can build \( f_\sigma, g_\sigma \) whose images under \( \recfnl{e'}{}{} \) agree.  

We now define (stagewise approximations to) \( f^{n}_\sigma \subfunneq f^{n+1}_\sigma ,  g^{n}_\sigma \subfunneq f^{n+1}_\sigma   \) and specify a computation for \( l_s \)  to ensure that \( f_\sigma = \Union_{n \in \omega} f^{n}_\sigma , g_\sigma = \Union_{n \in \omega} g^{n}_\sigma \)  are paths whose images agree.  Given any \( \sigma \in \dom T \) with \( T(\sigma) \) not extended by any \( f_\nu, g_\nu \), we'll define \( f^{0}_\sigma, g^{0}_\sigma \supfun T(\sigma) \)  with common initial segment \( T(\sigma) \).    When \( T_s(\sigma)\conv \neq T_{s+1}(\sigma) \) we stipulate that any \( \tau' = T_s(\sigma')  \) for \( \sigma' \)  less senior than \( \sigma \) or \( \sigma' = \sigma \) are injured.  If \( \tau \) is injured then we set any \( f^{n}_\sigma, g^{n}_\sigma  \supfuneq \tau  \) to be undefined and reset \( l_{s+1}(\tau) = \estr \).  Whenever \( f^{n}_\sigma \) (\( g^{n}_\sigma \)) is set to be undefined we also injure \( g^{n'}_\sigma  \) and \(  f^{n'}_\sigma  \) for  \(   n' \geq n \) ( \( f^{n'}_\sigma  \) and \(  g^{n'}_\sigma  \) for \( n' > n \), note the strict inequality here).  If we undefine  \( f^{0}_\sigma \) at stage \( s+1 \)  then we also injure \( T_s(\sigma) \) (the plus one ensuring that we injure the node \( f^{0}_\sigma \) is extending).   Finally, we also injure \( \tau \) immediately before defining \( f^{n}_\sigma \) or \( g^n_\sigma \) to be equal to \( \tau \).

At stage \( s \) we start with \( T_{s}(\estr) \) and act on nodes in \( \rng T_{s} \) in order of seniority  (which respects \( \subfun \)).    Suppose that at stage \( s \) we're taking action on \( \tau = T_{s}(\sigma)  \) after having dealt with all more senior and try the following cases in order.  Note that, at any time there will only be finitely many \( (n, \nu) \) with either \( f^{n}_\nu  \) or \( g^{n}_\nu \) defined so we can computably check whether a case is satisfied for some \( (n, \nu) \).  

\begin{pfcases*}
    \case[\( f^{n}_\nu = \tau \land f^{n+1}_\nu\diverge \) ]  If \( g^{n}_\nu \) is undefined then finish acting for \( \tau \) and all extensions of \( \tau \).  Otherwise, ensure  \( \recfnl{e'}{\tau}{}\concat l_{s+1}(\tau)  \supfun \recfnl{e'}{g^{n}_\nu}{} \).  If this is already satisfied by \( l_s(\tau) \) then let \( l_{s+1}(\tau) = l_s(\tau) \) otherwise pick the \( \subfun \) least solution.   We'll inductively ensure that \( \recfnl{e'}{g^{n}_\nu}{}  \supfun \recfnl{e'}{f^{n}_\nu}{}  \). 

    Now check if there is any \( \sigma' =  \sigma\concat[m] \) such that \( T_{s}(\sigma') = \tau' \) and \( \recfnl{e'}{\tau'}{} \supfun \recfnl{e'}{g^{n}_k}{}  \).  If so, take the most senior such \( \sigma' \) and define \( f^{n+1}_\nu = \tau' \) and continue to the next case.  If not, finish processing \( \tau \) for this stage.       

     \case[\( g^{n}_\nu = \tau \land g^{n+1}_\nu\diverge \) ] As above, using \( g^{n}_\nu, g^{n+1}_\nu \) in place of \( f^{n}_\nu, f^{n+1}_\nu \) and \( f^{n+1}_\nu \) in place of \( g^{n}_\nu \) (note the off by \( 1 \) difference).

     \case[\(  f^{0}_\sigma\conv \land g^{0}_\sigma\diverge \) ] We've already done all the hard work when defining \( f^{0}_\sigma \) so we look for the most senior \(\tau' = T_{s}(\sigma\concat[m]) \)  with \( \recfnl{e'}{f^{0}_\sigma}{} \subfuneq \recfnl{e'}{\tau'}{} \) and define \( g^{0}_\sigma = \tau' \).  If found, we move on to the next case. If not, finish processing \( \tau \) for this stage.

     \case[\( \forall(\nu,n)\left(\tau \neq f^{n}_\nu, g^{n}_\nu\right) \land f^{0}_\sigma\diverge \)]  Let \( \tau' =  T_{s}(\sigma\concat[m]) \) with \( \sigma\concat[m]\)  be the most senior extension of \( \sigma \) (if any).  If not found then finish processing \( \tau \) for this stage.  If found, set \( f^{0}_k = \tau' \) and set \( l_{s+1}(\tau)  \) to be the \( \subfun \) minimal value such that  \( \recfnl{e'}{\tau}{}\concat  l_{s+1}(\tau)  =  \recfnl{e'}{\tau'}{}  \).       

\end{pfcases*}

It is straightforward, to verify, that each node in \( \rng T \)  is only injured finitely many times and that each  \( f^{n}_\nu, g^{n}_\nu \), is only undefined/redefined finitely many times.  To see this, note first that once all more senior elements in \( \rng T \) have entered \( \rng T \) permanently the only way \( \tau \) is injured is if we define \( f^{n}_\nu, g^{n}_\nu \) to equal \( \tau \) or \( \tau = T_s(\nu) \) and \( f^0_\nu \) is injured.  It is clear from the construction that for at most one pair \( (n,\nu) \) and either \( f \) or \( g \) do we set \( f^{n}_\nu, g^{n}_\nu \) to equal \( \tau \) and that only if \( \tau \) goes unused in this way do we extend it by \( f^{0}_\nu \).  Thus, it is enough to show that each \( f^{n}_\nu, g^{n}_\nu \) eventually settles down.  To see this, note that if undefining \( h^{n}_\nu \) could cause \( \hat{h}^{n}_\nu \) to become undefined then \( h^{n}_\nu \) took the more senior extension of their common initial segment.  

It is also evident from the construction that the images of \( f_\nu \) and \( g_\nu \) under \( e' \) agree.  Specifically, that, when defined  \(\recfnl{e}{g^{n}_\nu}{} \supfun \recfnl{e}{f^{n}_\nu}{} \) and \(\recfnl{e}{f^{n+1}_\nu}{} \supfun \recfnl{e}{g^{n}_\nu}{} \).

We now show that for any \( \sigma  \in \dom T \) if \( \tau =  T(\sigma) \) isn't terminal than it is extended.  We first deal with the case where \( \tau \) isn't equal to any \( f^n_\nu \) or \( g^n_\nu \).  In this case, we clearly eventually define \( f^0_\sigma = T(\sigma\concat[m]) \) for some \( m \) and then for large enough \( m', t \) we will have \( \tau' =  T_{t}(\sigma\concat[m']) \) then \( \recfnl{e'}{\tau'}{} \supfun \recfnl{e'}{f^{0}_\sigma}{}  \) thus ensuring that we define \( g^0_\sigma \).

If \( \tau = f^{n+1}_\nu = T_{s+1}(\sigma) \) the agreement remarked on above ensures we eventually get a chance to define \( l_{s+1}(\tau)  \) so that \( \recfnl{e'}{f^{n}_\nu}{}\concat l_{s+1}(\tau) \supfun \recfnl{e'}{g^{n+1}_\nu}{}  \) after which \( \tau \)  is never injured.  Thus, for some sufficiently large \( m,t \) we have that if \( \tau' =  T_{t}(\sigma\concat[m]) \) then \( \recfnl{e'}{\tau'}{} \supfun \recfnl{e'}{g^{n}_\nu}{}  \) and we eventually define \( f^{n+2}_\nu \) to be equal to such a \( \tau' \).  The same argument applies, with the obvious adjustments to the indexes, when \( \tau = g^{n}_\nu \). 
 
Thus, if our tree is pruned then, as it is weakly, \( \omega \) branching for any \( \sigma \in \dom T \) there is some \( m \) with \( \sigma'=\sigma\concat[m] \) such that  \( f_{\sigma'}, g_{\sigma'} \) are defined and extend \( T(\sigma') \) with equal images under \( \recfnl{e}{}{} \).  The moreover claim is straightforward as well.      
\end{proof}

With slightly more care, we could ensure that every node \( \tau \in \rng T \) was extended by infinitely many paths \( f_{k, \tau}, g_{k, \tau} \).  We are unsure if this argument can be improved to construct a single total computable functional which witnesses all  \( \zeroj \) computable f-trees fail, in a strong sense, to have the properties necessary to help in a minimality style argument.

\section{Ordinal Notation Technicalities}

\label{app:ordnottech}

If we wish to build a sequence of trees \( T_n \) all homeomorphic to \( T_\omega \) ensuring the homeomorphism at limit levels only requires that we ensure \( T_n\restr{n} = T_\omega\restr{n} \) and that our homeomorphism from \( T_{n+1} \) to \( T_n \) is the identity on \( T_n\restr{n}  \).  In this appendix, we show that this idea can be extended to arbitrary ordinal notations.   The potential difficulty here is that we attempt to demand that  \( T_{\beta_n}\restr{n} \) should equal \( T_\lambda\restr{n} \) and that for some \(  \gamma \) with \( \beta_n \Oless \gamma \Oless \lambda \) we also demand \( T_\gamma\restr{m} = T_{\lambda'} \) for some \( \lambda' \Ogtr \lambda, m > n \).  

From here on out, we assume that we are working below some limit notation \( \alpha \) (e.g. \( T_\alpha \) is the top tree) and  show that for all \( \beta \Oleq \alpha \) we can effectively define \( \copyord{\beta} \) and \( \copylen{\beta} \).  The idea is that \( T_\beta \) will copy strings of length  \( \copylen{\beta} \) from the tree \( T_{\copyord{\beta}} \).   Note that, \( \copyord{\beta} \) and \( \copylen{\beta} \) will depend on both \( \alpha \) and \( \beta \).  We start by making the following definition.

\begin{definition}\label{def:minimal-e-path}
An \( \kleeneO\)-path   is a non-empty string \( \delta \in \wstrs  \) that satisfies the following for all  \( n + 1 \in \dom \delta, \kappa \in \kleeneO\) and some \( \beta \in \kleeneO\) 
\begin{enumerate}
    \item \( \delta(\lh{\delta} - 1) = \beta \)  
    \item If \( \delta(n) = \kappa + 1 \) then \( \delta(n+1) = \kappa \).
    \item\label{def:e-path:limit} If \( \delta(n) =  \lambda  \) and the notation \( \lambda \) isn't a successor  then \( \delta(n+1)  = \gamma \in \rng \Ofunc{\lambda}  \)
\end{enumerate}   
We call such an \( \kleeneO\)-path an \( \kleeneO\)-path from \( \delta(0) \) to \( \beta \).   The \( \kleeneO\)-path is minimal if the \( \gamma \) in \ref{def:e-path:limit} is required to be the minimal element in \( \rng \Ofunc{\lambda}   \) with \( \beta \Oleq \gamma \).   

\end{definition}

\begin{lemma}\label{lem:uniq-minimal-o-path}
If \( \beta \Oleq \alpha \) then there is a unique minimal \( \kleeneO\)-path from \( \alpha \) to \( \beta \).  Moreover, if  \( \delta \) is a minimal \( \kleeneO\)-path from \( \alpha \) to \( \beta \) then \( \delta\restr{n}, n >0 \) is the minimal \( \kleeneO\)-path from \( \alpha \) to \( \delta(n -1) \).   
\end{lemma}
\begin{proof}
Clearly, no initial segment of an \( \kleeneO\)-path from \( \alpha \) to \( \beta \) can be an \( \kleeneO\)-path from \( \alpha \) to \( \beta \) as an \( \kleeneO\)-path is a strictly decreasing sequence under \( \Oless \).  Moreover, 
 there is at most one minimal \( \kleeneO\)-path from \( \alpha \) to \( \beta \) as a minimal \( \kleeneO\)-path from \( \alpha \) to \( \beta \) is the lexicographically least \( \kleeneO\)-path (under \( \Oless \) ) from \( \alpha \) to \( \beta \).      This establishes the uniqueness and implies the moreover claim as well.  It only remains to show that there is always such an \( \kleeneO\)-path.

 Suppose not, then let \( f(0) = \alpha \) and define \( f(n+1) \) as per the definition of a minimal \( \kleeneO\)-path from \( \alpha \) to \( \beta \).  If we ever have \( f(n) = \beta \) then \( f\restr{n+1} \) is a minimal  \( \kleeneO\)-path from \( \alpha \) to \( \beta \).  If not, then, inductively, \( f(n) \Ogtr \beta \) and the conditions provide a unique definition for \( f(n+1) \).  Thus, \( f \) is an infinite descending sequence of notations.  Contradiction.     
\end{proof}

\begin{definition}\label{def:copylen-copyord}  
We inductively define \( \copylen{\beta}, \copyord{\beta} \) for \( \beta \Oless \alpha \) as follows. 

Let \( \delta \) be the minimal \( \kleeneO\)-path from \( \alpha \) to \( \beta \), \( n = \lh{\delta} \) and \( \gamma = \delta(n - 2) \).   We stipulate that \( \copylen{\alpha} =0 \) and break our definition into the following cases.
\begin{pfcases*}
    \case[\( \gamma \in \kleeneO+ \)]  Define \( \copylen{\beta} = \copylen{\gamma} \).  If \( \gamma \) is an even notation then define \( \copyord{\beta} = \gamma \) and otherwise define \( \copyord{\beta} =  \copyord{\gamma}  \).
    \case[\( \gamma \in \kleeneO- \)]  Let \( m \in \omega \) be such that \( \beta = \Ofunc{\gamma}(m) \).  Define \( \copyord{\beta} = \gamma \) and \( \copylen{\beta} = \copylen{\gamma} + m \).        
\end{pfcases*}

\end{definition}

\begin{lemma}\label{lem:limits-appear}
If  \( \lambda \in \kleeneO-, \lambda \Oleq \alpha \) and 
\[ \beta_m = \begin{cases}
                 \Ofunc{\lambda}(m) & \text{if } \Ofunc{\lambda}(m) \text{ is even} \\
                 \Ofunc{\lambda}(m) - 1 & \text{otherwise}
            \end{cases} \] 
then for all sufficiently large \( m \), \( \copyord{\beta_m} = \lambda \) and \( \copylen{\beta_m} \geq m \).   
\end{lemma}
\begin{proof}  
Let \( \delta_m \) be the minimal \( \kleeneO\)-path from \( \alpha \) to \( \beta_m  \).
We first establish that for all sufficiently large \( m \) we have \( \lambda \in \rng \delta_m \).  

Note that, by examination of the conditions from Definition \ref{def:minimal-e-path}, if \( \delta_n(k) = \kappa \Ogeq \lambda \) then we clearly have that \( \delta_n\restr{k+1} \subfun \delta_m \) for \( m \geq n \) (as \( \kappa \) clearly satisfies the minimality requirement for \( \delta_m \)).  Thus, if we see \( \lambda \in \rng \delta_n \) for any \( n \), we are done so suppose this never happens.   

Define \( f(k) \) to be equal to \( \delta_n(k) \) for the least value \( n \) such that \( \delta_{n}(k) \Ogtr \lambda \).  Obviously, \( f(0) \) is defined so suppose that \( f(k+1) \) fails to be defined for some minimal \( k \).  If \( f(k)  \) was a successor ordinal than \( \delta_{m}(k+1) = \lambda \) where \( m \) is the witness defining \( f(k) \) contradicting the supposition.  Thus, we can assume that \( f(k) = \kappa \in \kleeneO- \).  But, if \( \Ofunc{\kappa}(x) \Oleq \lambda \) for all \( x \) then we'd have \( \kappa \Oleq \lambda \) contra our assumption.  Thus, we can pick \( x \) maximal with \( \Ofunc{\kappa}(x) \Oless \lambda \) and then choose \( n \) so that \( \Ofunc{\lambda}(n) \Ogtr \Ofunc{\kappa}(x)  \).  Therefore, \( \delta_{n}(k+1) \Ogeq \Ofunc{\kappa}(x+1) \Ogeq  \lambda \) and, by assumption, we can't have equality showing that \( f(k+1) \) is defined.  But the function \( f \) defines an infinite decreasing sequence of notations.  Contradiction.  Therefore, for all sufficiently large \( m \)  we must have  \( \lambda \in \rng \delta_m \).  

If \( \lambda \in \rng \delta_m \) then we obviously have \( \beta_m = \delta(\lh{\delta} - 1) \) and either \( \lambda = \delta(\lh{\delta} - 2) \) if \( \beta_m = \Ofunc{\lambda}(m) \) or \( \lambda = \delta(\lh{\delta} - 3), \Ofunc{\lambda}(m) = \beta_m + 1 =  \delta(\lh{\delta} - 2) \).  In both cases, we clearly have  \( \copyord{\beta_m} = \lambda \) and  \( \copylen{\beta_m} \geq m \).  
\end{proof}

\begin{proposition}\label{prop:trees-copylen}
Suppose that for all \( \beta \Oleq \alpha \) we have  \( T_\beta\restr{\copylen{\beta}} = \tilde{T}_{\copyord{\beta}}\restr{\copylen{\beta}} \) where \( \tilde{T}_\kappa \supset T_\kappa \) for all \( \kappa \) then all of the following hold
\begin{enumerate}
    \item \label{prop:trees-copylen:comp} \( \copyord{\beta}, \copylen{\beta} \) are given by a computable function of \( \alpha, \beta \).

    \item\label{prop:trees-copylen:even} For all \( \beta \Oless \alpha \), \(  \copyord{\beta}  \) is an even notation satisfying \( \beta \Oless \copyord{\beta} \Oleq \alpha \)

    \item\label{prop:trees-copylen:between} If \( \beta \Oless \kappa \Oless \copyord{\beta}  \) then   \( T_\kappa\restr{\copylen{\beta}} \supset \tilde{T}_{\copyord{\beta}}\restr{\copylen{\beta}} \) and \( \copylen{\kappa} \geq \copylen{\beta} \).  Furthermore, if we always have \( \pruneTree{T_\beta}\restr{\copylen{\beta}} = \pruneTree{\tilde{T}}_{\copyord{\beta}}\restr{\copylen{\beta}} \) then \( \pruneTree{T_\kappa}\restr{\copylen{\beta}} = \pruneTree{\tilde{T}}_{\copyord{\beta}}\restr{\copylen{\beta}} \)

    \item\label{prop:trees-copylen:lim} If \( \lambda \Oleq \alpha \) is a limit notation then we can computably (in \( \lambda, n, \alpha \)) enumerate a sequence of even notations  \( \beta_n \Oless \lambda  \) with \( \copyord{\beta_n} = \lambda \) such that both \( \beta_n \)  and \( \copylen{\beta_n} \geq n \) are strictly monotonically increasing.   
\end{enumerate}
\end{proposition}
\begin{proof}
The computability is clear from the definition establishing \ref{prop:trees-copylen:comp}.   To verify \ref{prop:trees-copylen:even} note that if \( \copyord{\beta} \) was going to be an odd notation then we set \( \copyord{\beta} = \copyord{\gamma} \) which guarantees that  \( \copyord{\beta} \) is even.  As the only elements that can appear in a minimal \( \kleeneO\)-path from \( \alpha \) to \( \beta \) are \( \Oleq \) \( \alpha \) this establishes the other part of this claim.  For the remainder of the proof let

To verify \ref{prop:trees-copylen:between} suppose that \( \beta \Oless \kappa \Oless \copyord{\beta}  \),  \( \delta \) is a minimal \( \kleeneO\)-path from \( \alpha \) to \( \beta \) and \( \gamma = \delta(\lh{\delta} - 2) \).  WLOG we may assume that \( \gamma = \copyord{\beta} \) since otherwise \( \gamma \) is an odd notation and the claim for \( \beta \) follows by proving the claim for \( \gamma \).  If \( \delta' \) is a minimal \( \kleeneO\)-path from \( \alpha \) to \( \kappa \) then \( \delta' \supfun \delta^{-} \) since \( \delta^{-} \) is the minimal \( \kleeneO\)-path from \( \alpha \) to \( \gamma = \copyord{\beta} \Ogtr \kappa \).   

But if \( \kappa' \) appears in \( \rng \delta' \setminus \rng \delta^{-} \) then we must have \( T_{\kappa'}\restr{\copylen{\beta}} \supset \tilde{T}_{\gamma}\restr{\copylen{\beta}}  \) and \( \copylen{\kappa'} \geq \copylen{\beta} \)  (by induction on \( n \in \dom \delta' \setminus \dom \delta^{-} \)).  A similar argument shows the equality claim for the pruned trees under the provided assumptions.

Finally, for \ref{prop:trees-copylen:lim}  let  \( \beta'_n \) be the sequence defined in  Lemma \ref{lem:limits-appear} and inductively define \( \beta_{n+1} \) to be the first element  \( \beta'_m \) with \( \copyord{\beta'_m} = \lambda \) and (when \( n > 0 \)) \( \copylen{\beta'_m} > \copylen{\beta_n}, \beta'_m \Ogtr \beta_n  \).  Clearly, we can computably search for such an element and Lemma \ref{lem:limits-appear}  ensures that one will always be found.  The moreover claim follows immediately from \ref{prop:trees-copylen:between}.
\end{proof}

\begin{lemma}\label{lem:tree-uniformization}

Given a notation \( \lambda \in \kleeneO- \) and a tree   \(  T_\lambda \subset \wstrs   \) computable in \( \jumpn{X}{\lambda} \)  we can produce a tree \(   \tilde{T}_\lambda \supset  T_\lambda   \)  with \( [\tilde{T}_\lambda] = [T_\lambda] \) such that for all \( \beta \Oless \lambda \) if \( \copyord{\beta} = \lambda \) then \( \tilde{T}_\lambda\restr{\copylen{\beta}} \) is uniformly computable in \( \jumpn{X}{\beta} \).   Moreover, this holds with all possible uniformity and is agnostic to the index for \( T_\lambda \).   
\end{lemma}
Uniformly computable here specifically means there is a computable functional \( \Xi \) with \( \Xi(\beta, \jumpn{X}{\beta}, \sigma)  \) deciding the membership of \( \sigma \) in  \( \tilde{T}_\lambda \) when \( \copyord{\beta} = \lambda \) and \( \lh{\sigma} \leq \copylen{\beta} \) and that an index for \( \Xi \) is computable from \( \lambda \) and  an index for \( T_\lambda \).

\begin{proof}

By part \ref{prop:trees-copylen:lim} of Proposition \ref{prop:trees-copylen} let \( \beta_n \) be a monotonically increasing computable enumeration of the notations \( \beta \) with \( \copyord{\beta} = \lambda \) and \( l_n = \copylen{\beta_n} \) also monotonically increasing with limit \( \omega \).

 We now define \(  \tilde{T}_\lambda \) by setting \( \sigma \nin \tilde{T}_\lambda \) where \( n \) is the greatest with \( \lh{\sigma} \geq l_{n+1} \) just if we see a computation whose use is seen to be limited to \( \jumpn{X}{\beta_n} \) before stage \( n \) placing some \( \sigma' \subfun \sigma \) out of \( T_\lambda \).  Otherwise, place \( \sigma \) into \( \tilde{T}_\lambda \).  By seen to be limited to, we mean that the computation only consults columns \( \kappa \) of \( \jumpn{X}{\lambda} \) which are enumerated to be below some \( \beta_m, m \leq n \) in less than \( n \) steps.  

 Note that, if \( \lh{\sigma} < l_1 \) then \( \sigma \) is trivially in \(  \tilde{T}_\lambda \) and since if \( \copyord{\beta} = \lambda \) and \( \copylen{\beta} > \copylen{\beta_m} \) then \( \beta \Ogeq \beta_m \) we verify that  \( \tilde{T}_\lambda \) satisfies the uniform computation demands.  It remains only to verify that \( [ \tilde{T}_\lambda] = [T_\lambda] \).  Clearly, \( [\tilde{T}_\lambda] \supset [T_\lambda] \) since \( T_\lambda \subset \tilde{T}_\lambda \).  Now suppose that \( f \nin [T_\lambda] \) and let \( \sigma \subfun f \) such that \( \sigma \nin T_\lambda \).  Some finite computation from \( \jumpn{X}{\lambda} \) places \( \sigma \nin T_\lambda \) in \( s \) stages.  Since \( \jumpn{X}{\lambda} = \Tplus_{n \in \omega} \jumpn{X}{\Ofunc{\lambda}(n)} \) there is some \( n \) such that the use of this computation only mentions notations below \( \beta_n \).  Let \( t \geq s \) be a stage where every notation mentioned in the use of this computation has been enumerated to be below \( \beta_n \).  Now if \( \lh{\tau} = l_{t+1}, \tau \supfun \sigma \) then \( \tau \nin T_\lambda \).

\end{proof}  

Finally, we note that not only is the limit lemma completely uniform and relativeizeable we can ensure that when we iterate it's application only the final limit is at risk of being undefined.


\begin{lemma}\label{lem:limit-lemma}
Given a functional \( \Upsilon_2 \), we can uniformly construct a computable functional \( \rho^{X}(z, s_1, s_0) \)   such that if \( \rho^{X}_{1}(z, s_1) \eqdef \lim_{s_0 \to \infty} \rho^{X}_n(z, s_1, s_0)  \) then 
\begin{itemize}
    \item For all \( X\) \( \rho^{X}_1, \rho^{X} \) are total binary valued functions.
    \item \( \Upsilon_2(\jjump{X};z)\conv \in \set{0,1}{} \) iff \( \rho_2^{X}(z)\conv = \Upsilon_2(\jjump{X};z) \).
\end{itemize}
Moreover,  \( \rho^{X}(z,s_1,s_0) \) depends only on the computations of the form \( \Upsilon_2(\tau;z) \) for some \( \tau \). 
   
\end{lemma}
\begin{proof}
Using the limit lemma relativized to \( \jump{X} \)  we can derive a total computable binary valued functional whose limit gives \( \Upsilon_2(\jjump{X};z) \) when \( \Upsilon_2(\jjump{X};z)\conv \in \set{0,1}{} \)  and, by alternating between \( 0,1 \) whenever we haven't settled on a computation witnessing  \( \Upsilon_2(\jjump{X};z)\conv \in \set{0,1}{} \), ensures that otherwise the limit fails to exist. Using a stagewise approximation to \( \jump{X} \) gives us our functional \( \rho^{X} \) and ensures \( \rho^{X}_1 \) is total.  The moreover, follows by attention to how the relativization was performed.
\end{proof}

\bibliographystyle{plain}
\bibliography{Remote}

\begin{thebibliography}{10}

\bibitem{Andrews2014Degrees}
Uri Andrews, Peter Gerdes, and Joseph~S. Miller.
\newblock The degrees of bi-hyperhyperimmune sets.
\newblock {\em Annals of Pure and Applied Logic}, 165(3):803--811, 2014.

\bibitem{Friedman1975One}
Harvey Friedman.
\newblock One hundred and two problems in mathematical logic.
\newblock {\em The Journal of Symbolic Logic}, 40(02):113--129, 1975.

\bibitem{Gerdes2010HarringtonS}
Peter~M. Gerdes.
\newblock Harrington's solution to {McLaughlin's} conjecture and non-uniform
  self-moduli.
\newblock Unpublished preprint, 2010.

\bibitem{Harrington1976Arithmetically}
Leo~A Harrington.
\newblock Arithmetically incomparable arithmetic singletons.
\newblock Unpublished mimeographed notes, 1976.

\bibitem{Harrington1976MclaughlinS}
Leo~A Harrington.
\newblock {Mclaughlin's} conjecture.
\newblock Unpublished handwritten notes, 1976.

\bibitem{Jockusch1984PseudoJump}
Carl~G. Jockusch and Richard~A. Shore.
\newblock Pseudo-jump operators. ii: Transfinite iterations, hierarchies and
  minimal covers.
\newblock {\em The Journal of Symbolic Logic}, 49(04):1205--1236, 1984.

\bibitem{Kurtz1981Randomness}
Stuart~A. Kurtz.
\newblock {\em Randomness and Genericity in the Degrees of Unsolvability}.
\newblock Ph.d., University of Illinois at Urbana-Champaign, United States --
  Illinois, 1981.

\bibitem{Lachlan1966Lower}
A.~H. Lachlan.
\newblock Lower bounds for pairs of recursively enumerable degrees.
\newblock {\em Proceedings of the London Mathematical Society},
  s3-16(1):537--569, 1966.

\bibitem{Odifreddi1992Classical}
Piergiorgio Odifreddi.
\newblock {\em Classical recursion theory: Volume 2}, volume~2.
\newblock Elsevier, 1992.

\bibitem{Sacks1964Recursively}
Gerald~E. Sacks.
\newblock The recursively enumerable degrees are dense.
\newblock {\em The Annals of Mathematics}, 80(2):300--312, 1964.

\bibitem{Sacks1971Forcing}
Gerald~E Sacks.
\newblock Forcing with perfect closed sets.
\newblock In {\em Axiomatic Set Theory}, volume~1 of {\em Proceedings of
  Symposia in Pure Mathematics}, pages 331--355. American Mathematical Society,
  1971.

\bibitem{Sacks1990Higher}
Gerald~E. Sacks.
\newblock {\em Higher Recursion Theory}.
\newblock Perspectives in Logic. Cambridge University Press, Cambridge, 2017.

\bibitem{Shoenfield1959On}
J.~R. Shoenfield.
\newblock On degrees of unsolvability.
\newblock {\em Annals of Mathematics}, 69(3):644--653, 1959.

\bibitem{Simpson1985Arithmetic}
Mark~Fraser Simpson.
\newblock {\em Arithmetic Degrees: Initial Segments, {$\omega$}-REA Operators
  and the {$\omega$}-jump}.
\newblock Phd thesis, Cornell, 1985.

\bibitem{Simpson2016Implicit}
Stephen~G. Simpson.
\newblock Implicit definability in arithmetic.
\newblock {\em Notre Dame Journal of Formal Logic}, 57(3):329--339, 2016.

\end{thebibliography}

\end{document}